\documentclass{article}
\usepackage{amsmath}
\usepackage{amsfonts}
\usepackage{amsthm}
\usepackage{enumerate}
\usepackage{amssymb}

\newtheorem{thm}{Theorem}[section]
\newtheorem{prop}[thm]{Proposition}
\newtheorem{cor}[thm]{Corollary}
\newtheorem{lem}[thm]{Lemma}

\newtheorem*{thmnn}{Theorem}

\theoremstyle{remark}

\theoremstyle{definition}
\newtheorem{defn}[thm]{Definition}

\begin{document}

\title{The Structure of Integral Parabolic Subgroups of Orthogonal Groups}

\author{Shaul Zemel}

\maketitle

\section*{Introduction}

The structure of semi-simple linear algebraic groups over fields of characteristic 0, as well as of their parabolic subgroups, is well-known for quite some time. The latter admits a \emph{Levi decomposition}: It contains a normal unipotent radical, the quotient is semi-simple, and the resulting short exact sequence splits. Therefore the parabolic subgroup itself is a semi-direct product, and the image of any splitting of this sequence is a \emph{Levi subgroup}. The considered in this paper is where the group is the orthogonal group of a non-degenerate quadratic space $V$ and the parabolic group is maximal, which is therefore associated with a non-trivial isotropic subspace $U$ of $V$. All the parabolic subgroups associated with isotropic subspaces of the same dimension are conjugate, but those associated with subspaces of different dimensions are not isomorphic.

While this description is easy and simple over any field of characteristic 0, the structure becomes more complicated when one considers subgroups that are defined over rings that are not fields. In particular, arithmetic subgroups can be viewed as analogues of these groups over $\mathbb{Z}$, and the theory of the parabolic subgroups there, in particular explicit splitting of the short exact sequence of the Levi decomposition, can become much more involved. The goal of this paper is to investigate this question in the setting where the group is the orthogonal group of $V$ as above, where the underlying field is $\mathbb{Q}$, and the arithmetic group is the discriminant kernel of an even lattice $L$ inside $V$.

As an application, we recall that in some cases the symmetric spaces of real linear algebraic groups can be Hermitian, and that by \cite{[BB]}, quotients of the Hermitian symmetric spaces by arithmetic subgroups are complex algebraic varieties, which are defined over number fields, and are also known as \emph{Shimura varieties}. Knowledge about the maximal parabolic subgroups gives information about the canonical Baily--Borel compactifications (see also \cite{[S]}), and combining this with the theory of toric varieties can be used, as in \cite{[AMRT]} and \cite{[Nam]}, to the construction of toroidal compactifications. However, the precise form of the boundary components of these compactifications depends on the structure of the parabolic subgroups of the arithmetic group in question. In the orthogonal case, the symmetric spaces is Hermitian only if the signature of $V$ over $\mathbb{R}$ is $(n,2)$ (or $(n,2)$, which is dual to it), where there are at most two types of maximal parabolic subgroups, because the dimension of a rational isotropic subspace can be at most 2 in this case.

Now, the structure of both the Baily--Borel and the toroidal compactifications is described in \cite{[F]}, where the toroidal boundary components arising from 2-dimensional isotropic subspaces are canonical. We recall that the structure over $\mathbb{Z}$ stays simple in case the isotropic subspace has dimension 1 (this appears already implicitly in \cite{[Bo]}), but the delicacies appearing in the case of dimension 2 were overlooked in \cite{[F]}. The second goal of this paper is to use the precise structure of the corresponding parabolic subgroup proven here in order to give the exact form of the associated canonical toroidal boundary components.

\smallskip

We can now present the results of this paper in more detail. Let $\mathrm{O}(V)$ be the orthogonal group associated with the rational quadratic space $V$, and let $\mathcal{P}_{U}$ be the parabolic subgroup corresponding to the isotropic subspace $U$ of $V$, where we set $W=U^{\perp}/U$. Then the Levi quotient is the product of $\mathrm{O}(W)$ and $\mathrm{GL}(U)$, and the unipotent radical $\mathcal{W}_{U}$ is the Heisenberg group associated with the anti-symmetric bilinear map taking a pair of homomorphisms $\psi$ and $\varphi$ in $\mathrm{Hom}(W,U)$ to the homomorphism $\frac{\psi\varphi^{*}-\varphi\psi^{*}}{2}\in\mathrm{Hom}^{as}(U^{*},U)$ (where the superscript denotes anti-symmetry). It follows that if $\dim U=1$ then $\mathcal{W}_{U}$ is Abelian, and otherwise it is nilpotent of class 2 (in fact, it can be verified that the nilpotence class of the unipotent radical of every maximal parabolic subgroup of a classical group is either 1 or 2). The choice of a Levi subgroup can be done is several equivalent ways, and one of them is by intersecting with a complementary maximal parabolic subgroup, which is associated with an isotropic vector space of the same dimension as $U$. More explicitly, we take a complement $\tilde{U}$ for $U^{\perp}$ in $V$, and then the Levi subgroup is the intersection of the stabilizers of $U$ and of $U\oplus\tilde{U}$, and every element of $\mathcal{P}_{U}$ can be written uniquely as a product of an element of $\mathcal{W}_{U}$ and a Levi element. A crucial observation here is that $\tilde{U}$ need not be isotropic, and this is important since good complements over $\mathbb{Z}$ do not always exist. We also remark that the factor $\frac{1}{2}$ appearing in the definition of the map used for the Heisenberg group indicates that care has to be taken when working over $\mathbb{Z}$. Two incarnations of this fact are the definition of semi-characters (rather than characters) in Section 2.2 of \cite{[BL]}, and the condition about vanishing of Fourier coefficients of generalized Jacobi forms following Definition 2 of \cite{[W]} (note that the published version of this paper contains a mistake, but it is corrected in the arXiv version).

As mentioned above, the arithmetic group arises from a lattice $L$ in $V$, and we take it to be the discriminant kernel, also called the stable orthogonal group, of $L$, which we denote by $\Gamma_{L}$. Moreover, in all the parabolic groups we consider only those elements which lie in the connected component of the corresponding real Lie group, which yields nicer results. We set $I=U \cap L$ and $\Lambda=(I^{\perp} \cap L)/I \subseteq W$, and we denote the group $\Gamma_{L}\cap\mathcal{P}_{U}$ in which we are interested by $\Gamma_{L,I}$. We observe that for having a good decomposition of $L$ and its dual lattice $L^{*}$ we must consider only complements $\tilde{U}$ whose intersection with $L^{*}$ yields a dual for $I$ over $\mathbb{Z}$ (this is the reason for allowing non-isotropic complements, as isotropic complements with this property do not always exist). Now, as already observed in \cite{[Nak]}, \cite{[Bo]}, \cite{[Br]}, \cite{[Z1]}, \cite{[Z2]}, and others, when $I$ has rank 1 the intersection of $\Gamma_{L}$ with the unipotent radical $\mathcal{W}_{U} \cong U$ is just $\Lambda$, the image in the Levi quotient is the discriminant kernel $\Gamma_{\Lambda}$, and $\Gamma_{L,I}$ is the resulting semi-direct structure also over $\mathbb{Z}$, regardless of the choice of the complement $\tilde{U}$.

On the other hand, when $\dim U\geq2$ the structure is more complicated, and to describe it we need some additional notions. First, the Heisenberg group $\mathcal{W}_{U}$ is now a non-trivial extension, and in defining its integral counterpart, with center $\mathrm{Hom}^{as}(I^{*},I)$ and quotient $\mathrm{Hom}(\Lambda^{*},I)$, the elements mapping to some particular homomorphism from $\Lambda^{*}$ to $I$ may have to lie in a non-trivial coset inside $\mathrm{Hom}^{as}\big(I^{*},\frac{1}{2}I\big)$ for the group law to work. In addition, there are two subgroups of $U$, both of which contain $I$ with finite index, that we denote by $I_{L^{*}}$ and $I_{\iota}$, and we denote by $\mathrm{SL}(I_{L^{*}},I)$ and by $\mathrm{SL}(I_{\iota},I)$ the congruence subgroups of $\mathrm{SL}(I)$ that acts trivially on $I_{L^{*}}/I$ or on $I_{\iota}/I$. Then the former group contains the latter, and we recall that the quotient $\overline{\mathcal{P}}_{U}$ obtained by dividing $\mathcal{P}_{U}$ by $\mathrm{Hom}^{as}(U^{*},U)$ is the semi-direct product of the Levi subgroup and $\mathrm{Hom}(W,U)$. We can now describe, in the choice of coordinates arising from $\tilde{U}$, the intersection of $\Gamma_{L}$ with $\mathcal{W}_{U}$ and the image $\overline{\Gamma}_{L,I}$ of $\Gamma_{L,I}$ in $\overline{\mathcal{P}}_{U}$ as follows.
\begin{thmnn}
The intersection $\Gamma_{L}\cap\mathcal{W}_{U}$ is precisely the integral Heisenberg group mentioned above. In addition, there is a cocycle of $\mathrm{SL}(I_{L^{*}},I)$ with values in $\mathrm{Hom}(\Lambda,I)/\mathrm{Hom}(\Lambda^{*},I)$ such that $\overline{\Gamma}_{L,I}$ consists of those triples of $M\in\mathrm{SL}(U)$, $\gamma\in\mathrm{O}(W)$, and $\psi\in\mathrm{Hom}(W,U)$ for which is in $\mathrm{SL}(I_{L^{*}},I)$, $\gamma$ is in $\Gamma_{\Lambda}$, and $\psi$ is in $\mathrm{Hom}(\Lambda,I)$ such that the image of $\psi$ modulo $\mathrm{Hom}(\Lambda^{*},I)$ is the image of $M$ under the cocycle map.
\end{thmnn}
This theorem is the combination of Theorem \ref{sdprodZ}, Proposition \ref{actcoset}, and some results of Proposition \ref{paragrpZ}. We note that the latter theorem gives the full description of $\Gamma_{L,I}$ in these coordinates, but the formula for the remaining coordinate, from $\mathrm{Hom}^{as}(U^{*},U)$, is more complicated, and also involves information from the possible lack of isotropy of $\tilde{U}$. We do mention that this coordinate does not depend on the coordinate $\gamma\in\Gamma_{\Lambda}$.

\smallskip

For the application to toroidal boundary components, we first observe that when $\dim U=2$ the space $\mathrm{Hom}^{as}(U^{*},U)$ is 1-dimensional, so that having the lattice $I$ in $U$ and an orientation identifies this space with $\mathbb{Q}$. Now, the symmetric space of the connected component $\mathrm{SO}^{+}(V)$ when $V$ is of signature $(n,2)$ is a complex manifold of dimension $n$, which the choice of $U$ and a basis for $I \subseteq U$ presents as a double fibration. The base space of both fibration is the upper half-plane $\mathcal{H}$, the first fibration over it is the tangent of the universal elliptic curve tensored with $\Lambda$, and the last one is essentially an additional coordinate from $\mathcal{H}$. As for the action of $\Gamma_{L,I}$, the coordinate from $\mathrm{Hom}^{as}(U^{*},U)\cong\mathbb{Z}$ operates on the upper fibers by translations, the one from $\mathrm{Hom}(\Lambda^{*},I)\cong\Lambda\oplus\Lambda$ to translations on the fiber $\Lambda_{\mathbb{C}}$ of the basic fibration, the group $\Gamma_{\Lambda}$ (which is finite since $\Lambda$ is positive definite) also acts on $\Lambda_{\mathbb{C}}$, and finally the coordinate from $\mathrm{SL}(I_{L^{*}},I)$, which is now a congruence subgroup of $\mathrm{SL}_{2}(\mathbb{Z})$, operates on the base space $\mathcal{H}$. Dividing by $\mathrm{Hom}^{as}(U^{*},U)$ the coordinate of the upper fiber becomes a punctured disc coordinate, and the toroidal boundary component is obtained by completing the puncture of the disc. Now, the quotient of the universal family by $\mathrm{Hom}(\Lambda^{*},I)$ produces the universal elliptic curve over $\mathcal{H}$ tensored with $\Lambda$, and dividing by the small congruence subgroup $\mathrm{SL}(I_{\iota},I)$ gives an open Kuga--Sato variety. Recalling the remaining parts of $\Gamma_{L,I}$, we obtain the following result.
\begin{thmnn}
The boundary component in question is a finite quotient of a Kuga--Sato variety, by both the cocycle relation and $\Gamma_{\Lambda}$.
\end{thmnn}
The precise form appears in Theorem \ref{boundcomp}. We remark that the disc coordinate defining this boundary component can be multiplied by rather non-trivial roots of unity in the process, though its vanishing locus is clearly unaffected.

\smallskip

The paper is divided into 4 sections. Section \ref{ParSubgp} presents the well-known form of the parabolic subgroups of orthogonal groups over fields, mainly to introduce the necessary notation. Section \ref{Lattices} shows how lattices and their duals decompose in suitably chosen coordinates, and Section \ref{ArithSbgp} then describes the structure of the parabolic subgroups over $\mathbb{Z}$. Finally, Section \ref{TorComp} applies these results for determining the form of the canonical boundary components in the toroidal compactifications of the corresponding orthogonal Shimura varieties.

\section{Parabolic Subgroups of Orthogonal Groups \label{ParSubgp}}

In this section we present the structure of an arbitrary maximal parabolic subgroup of an orthogonal group of a vector space over a field of characteristic different from 2. All of the results here are well-known, and can be deduced in a standard manner from the general theory involving roots, appearing in the classical books \cite{[Bor]} and \cite{[H]} (among others). We give them here in order to fix the notation and because they serve as a basis for the investigation over $\mathbb{Z}$ in the later sections.

\medskip

Let $V$ be a non-degenerate finite-dimensional quadratic space over a field $\mathbb{F}$ of characteristic different from 2. We shall denote the image of two vectors $u$ and $v$ in $V$ under the associated bilinear form by $(u,v)$, and we shorthand $(v,v)$ to $v^{2}$ for every $v \in V$. The quadratic form on $V$ therefore sends $v \in V$ to $\frac{v^{2}}{2}$, and we have the \emph{orthogonal group} \[\mathrm{O}(V):=\mathrm{Aut}_{\mathbb{F}}\big(V,(\cdot,\cdot)\big)=\{A\in\mathrm{GL}(V)|\;(Au,Av)=(u,v)\mathrm{\ for\ all\ }u\mathrm{\ and\ }v\mathrm{\ in\ }V\}.\] Given a subspace $U$ of $V$, as well as an arbitrary space $X$ over $\mathbb{F}$, we write \[U^{\perp}=\{v \in V|\;(u,v)=0\mathrm{\ for\ every\ }u \in U\}\quad\mathrm{and}\quad X^{*}:=\mathrm{Hom}_{\mathbb{F}}(X,\mathbb{F})\] for the subspace orthogonal to $U$ and the space dual to $X$ respectively. The bilinear form induces an isomorphism
\begin{equation}
V^{*} \cong V,\mathrm{\ and\ more\ generally\ }U^{\perp} \cong V/U^{\perp}\mathrm{\ for\ every\ subspace\ }U \subseteq V. \label{isobil}
\end{equation}

We recall that $U$ is called \emph{isotropic} when
\[(u,w)=0\mathrm{\ for\ all\ }u\mathrm{\ and\ }v\mathrm{\ in\ }U \Longleftrightarrow U \subseteq U^{\perp},\] and then we set \[\mathcal{P}_{U}:=\mathrm{St}_{\mathrm{O}(V)}U \quad\mathrm{as\ well\ as}\quad W:=U^{\perp}/U\] to be the \emph{parabolic subgroup} of $\mathrm{O}(V)$ that is associated with $U$ and the corresponding non-degenerate quadratic space respectively, the latter being of dimension $\dim V-2\dim U$. These are all the maximal parabolic subgroups of $\mathrm{O}(V)$, and we are interested in their structure.

Now, since an element $\mathcal{P}_{U}$ must also preserve $U^{\perp}$ and therefore act on the quotient $W$, we immediately obtain the following first result.
\begin{lem}
The parabolic group $\mathcal{P}_{U}$ comes with a natural map into the product $\mathrm{GL}(U)\times\mathrm{O}(W)$. \label{unipker}
\end{lem}
Explicitly, the map takes $A\in\mathcal{P}_{U}$ to the pair consisting of $M=A\big|_{U}\mathrm{GL}(U)$ and of the map $\gamma$ on $W$ that is induced from $A\big|_{U^{\perp}}$, where the latter lies in $\mathrm{O}(W)$ because $A\big|_{U^{\perp}}$ preserves the bilinear form on $U^{\perp}$ and hence also on $W$. Note that there is another quotient on which $A$ acts, which is $V/U^{\perp}$, and is identified with $U^{*}$ via Equation \eqref{isobil}. For determining this we take $u \in U$ and $v \in V$ such that $v+U^{\perp}$ corresponds to $v^{*} \in U^{*}$ and write \[(v^{*},u)=(v,u)=(Av,Au)=(Av,Mu)=\big((Av)^{*},Mu\big)=\big(M^{*}(Av)^{*},u\big),\] where $M^{*}\in\mathrm{GL}(U^{*})$ is the map that is dual to $M$. It follows the action of $A$ on $V/U^{\perp} \cong U^{*}$ must be like that of the inverse $M^{-*}=(M^{*})^{-1}$ of $M^{*}$, and therefore produces no new information.

It turns out, as we shall soon see, that the quotient from Lemma \ref{unipker} is the Levi quotient of $\mathcal{P}_{U}$, and the \emph{unipotent radical} of $\mathcal{P}_{U}$ is the kernel $\mathcal{W}_{U}$ of the map from that lemma. In order to analyze it, we first observe that Equation \eqref{isobil} gives $W^{*} \cong W$ (since $W$ is also non-degenerate), which allows us to view the dual of every map $\psi:W \to U$ as a map $\psi^{*}:U^{*} \to W$. In addition, a map $\eta:U^{*} \to U$ is called \emph{symmetric} (resp. \emph{anti-symmetric}) if the bilinear form \[(u^{*},v^{*}) \in U^{*} \times U^{*}\mapsto(u^{*},\eta v^{*}) \in U^{*} \times U\mapsto\mathrm{\ the\ pairing\ }(u^{*},\eta v^{*})\in\mathbb{F}\] is symmetric (resp. anti-symmetric) under interchanging the two variables. We denote the space of anti-symmetric linear maps from $U^{*}$ to $U$ by $\mathrm{Hom}_{\mathbb{F}}^{as}(U^{*},U)$.

We also recall the following definition.
\begin{defn}
Let $X$ be a vector space over $\mathbb{F}$ supplied with an anti-symmetric bilinear map $B$ to another vector space $Y$ over $\mathbb{F}$. Then the associated \emph{Heisenberg group} is defined to be \[H(X,Y):=X \times Y\mathrm{\ with\ the\ product\ rule\ }(x,y)\cdot(\xi,\eta)=\big(x+\xi,y+\eta+B(x,\xi)\big).\] On the other hand, if a space $Z$ carries a $Y$-valued \emph{symmetric} bilinear map $(z,\zeta) \mapsto \langle z,\zeta \rangle \in Y$, then on $Z \times Z$ we have the anti-symmetric bilinear map \[\big((z,w),(\zeta,\omega)\big)\mapsto\langle z,\omega \rangle-\langle w,\zeta \rangle,\mathrm{\ and\ we\ denote\ }\tilde{H}(Z,Y):=H(Z \times Z,Y).\] \label{Heisdef}
\end{defn}
The group $H(X,Y)$ from Definition \ref{Heisdef} lies in a short exact sequence
\begin{equation}
0 \to Y \to H(X,Y) \to X \to 0, \label{Heisseq}
\end{equation}
where $Y$ is contained in the center of $H(X,Y)$. When $B$ is non-degenerate, $Y$ is the full center of $H(X,Y)$. Since $Y$ is central and $X$ is commutative, the commutator of the pair $(x,y)$ and $(\xi,\eta)$ lies in $Y$ and depends only on $x$ and $\xi$---indeed, it equals $\big(0,2B(x,\xi)\big)$. In fact, we can now prove that this condition determines $H(X,Y)$ as an extension of $X$ by $Y$.
\begin{prop}
Let $H$ be a group mapping onto space $X$ with central kernel $Y$ as in Equation \eqref{Heisseq}, and assume that when $h$ and $k$ are elements of $H$, with $X$-images $x$ and $\xi$ respectively, then the commutator $[h,k]$ is $2B(x,\xi) \in Y$ for some anti-symmetric $Y$-valued bilinear form on $X$. Then $H \cong H(X,Y)$ (with this $B$) as extensions of $X$ by $Y$. \label{Heischar}
\end{prop}
Note that as $X$ and $Y$ are commutative while $H$ is not, we write the group operation of $H$ multiplicatively while the operations on $X$ and $Y$ additively, so that on $Y \subseteq H$ addition and multiplication have the same meaning.

\begin{proof}
Choose a basis $(u_{i})_{i \in I}$ for $X$ as well as some linear order on the set of indices $I$, and lift these basis elements arbitrarily to $H$. Since elements of $H$ mapping to multiples of the same vector in $X$ commute (by the anti-symmetry of $B$), our lifts generate a one-parameter subgroup of $H$ for every $i \in I$, mapping to the corresponding 1-dimensional vector spaces of $X$. Denote the resulting lift of $cu_{i} \in X$ by $h_{cu_{i}} \in H$, and recall that the presentation of any element of $X$ involves only finitely many basis elements. We can therefore define the lift \[x=\sum_{j=1}^{n}c_{j}u_{i_{j}}\mapsto\textstyle{h_{x}:=h_{c_{1}u_{i_{1}}} \ldots h_{c_{n}u_{i_{n}}}\Big/\sum_{j<k}B(c_{j}u_{i_{j}},c_{k}u_{i_{k}})},\] where we order the appearing indices such that $i_{j}<i_{k}$ wherever $j<k$ (this is well-defined by the centrality of $Y$). Since $H$ lies in a short exact sequence like the one from Equation \eqref{Heisseq}, every element of $H$ is of the form $h_{x}y$ for some $x \in X$ and $y \in Y$, and we claim that the map sending this element to $(x,y) \in H(X,Y)$ yields the desired isomorphism. By this formula and the centrality of $Y$, it suffices to verify that for $x$ and above and another element, say $\xi=\sum_{j=1}^{n}d_{j}u_{i_{j}} \in X$ (where we may trivially increase the set $\{u_{i_{j}}\}_{j=1}^{n}$ spanning $x$ to span also $\xi$), we have \[h_{x} \cdot h_{\xi}=h_{x+\xi} \cdot B(x,\xi),\mathrm{\ with\ }B(x,\xi) \in Y\] (then our map will already be an isomorphism as extensions). But indeed, the left hand side is \[\textstyle{h_{c_{1}u_{i_{1}}} \ldots h_{c_{n}u_{i_{n}}}h_{d_{1}u_{i_{1}}} \ldots h_{d_{n}u_{i_{n}}}\Big/\sum_{j<k}B(c_{j}u_{i_{j}},c_{k}u_{i_{k}})B(d_{j}u_{i_{j}},d_{k}u_{i_{k}})}\] (by the centrality of $Y$), which equals \[h_{(c_{1}+d_{1})u_{i_{1}}} \ldots h_{(c_{n}+d_{n})u_{i_{n}}}\sum_{j<k}\big[2B(c_{k}u_{i_{k}},d_{j}u_{i_{j}})-B(c_{j}u_{i_{j}},c_{k}u_{i_{k}})B(d_{j}u_{i_{j}},d_{k}u_{i_{k}})\big]\] by the commutation relations, and the latter expression yields the desired right hand side by the definition of $h_{x+\xi}$ and the anti-symmetry and bi-additivity of $B$. This proves the proposition.
\end{proof}
We remark that the anti-symmetry of $B$ and the commutation relation also imply that the form of the lift of $h_{x}$ in the proof of Proposition \ref{Heischar} does not depend on the ordering of the chosen basis $(u_{i})_{i \in I}$ or on the product we take for defining $h_{x}$, as long as the ordering of the indices coincides with the order in the product. In addition, one easily verifies that in \emph{any} group $H$ lying in a short exact sequence as in Equation \eqref{Heisseq} with $Y$ central, the commutator map factors through a map from $X \times X$ to $Y$, which must be bi-additive and anti-symmetric. Hence every such extension over $\mathbb{Q}$ is a Heisenberg group (by Proposition \ref{Heischar}), and the same happens for \emph{continuous} extensions over $\mathbb{R}$ or over $\mathbb{Q}_{p}$, but in general such extensions may involve anti-symmetric bi-additive maps that are not bilinear (e.g., semi-bilinear maps involving Galois automorphisms). Proposition \ref{Heischar} may also be proved by general arguments relating group extensions to the cohomology of one Abelian group acting trivially on another Abelian group, under appropriate conditions (in particular, the action of 2 on the second group must be invertible).

The basic structure of the group $\mathcal{W}_{U}$ can now be described.
\begin{prop}
The unipotent radical $\mathcal{W}_{U}$ is a subgroup of the Heisenberg group $H\big(\mathrm{Hom}_{\mathbb{F}}(W,U),\mathrm{Hom}_{\mathbb{F}}^{as}(U^{*},U)\big)$, where the anti-symmetric map from $\mathrm{Hom}_{\mathbb{F}}(W,U)$ into $\mathrm{Hom}_{\mathbb{F}}^{as}(U^{*},U)$ sends $\psi$ and $\varphi$ in the former space to $\frac{\psi\varphi^{*}-\varphi\psi^{*}}{2}$ in the latter. The normal subgroup $\mathrm{Hom}_{\mathbb{F}}^{as}(U^{*},U)$ of this Heisenberg group is contained in $\mathcal{W}_{U}$. \label{unipstruc}
\end{prop}

\begin{proof}
We first construct a map from $\mathcal{W}_{U}$ to $\mathrm{Hom}_{\mathbb{F}}(W,U)$. Recall that an element $A\in\mathcal{W}_{U}$ acts trivially on $U$, so that the map sending $v \in V$ to $v-Av$ is well-defined on $V/U$. Moreover, since the restriction of $A\big|_{U^{\perp}}$ operates trivially on $W=U^{\perp}/U$, we deduce that the map
\begin{equation}
\psi:W \to U,\quad\psi:w=z+U \mapsto z-Az \label{psidef}
\end{equation}
is well-defined and indeed takes values in $U$. It is easily verified that the map from $\mathcal{W}_{U}$ to $\mathrm{Hom}_{\mathbb{F}}(W,U)$ defined by $A\mapsto\psi$ is a homomorphism of groups.

Now, since $A$ is an isometry with $A\big|_{U}=\mathrm{Id}_{U}$, it acts trivially also on $V/U^{\perp}$, and therefore $Av-v \in U^{\perp}$ for every $v \in V$. Combining this with the isometry property of $A$, we obtain for $v \in V$ and $z \in U^{\perp}$ the equality \[0=(Av,Az)-(v,z)=(Av-v,z)+(v,Az-z)=(Av-v,z)-\big(v,\psi(z+U)\big)\] (since $Av-v \in U^{\perp}$ is perpendicular to $Az-z \in U$ and by Equation \eqref{psidef}). Now, the fact that $Av-v \in U^{\perp}$ allows us to replace $z$ in the latter equation by $w=z+U \in W$, and we observe that the image $Av-v+U \in W$ of $Av-v$ depends only on the image of $v$ in $V/U^{\perp}$ (indeed, replacing $v$ by $v+z$ for $z \in U^{\perp}$ changes $Av-v$ only by $Az-z \in U$, and take the image modulo $U$). Recalling the isomorphism the isomorphism from Equation \eqref{isobil}, we obtain a map $\hat{\psi}:U^{*} \to W$ that sends $v^{*} \in U^{*}$ to the image $Av-v$ for the corresponding element $v+U^{\perp} \in V/U^{\perp}$, and the latter equation then gives
\begin{equation}
(\hat{\psi}v^{*},w)=(Av-v,w)=(v,\psi w)=(v+U^{\perp},\psi w)=(v^{*},\psi w), \label{psidual}
\end{equation}
so that $\hat{\psi}$ must be the map $\psi^{*}$ that is dual to $\psi$.

Next, note that $\ker\big(\mathcal{W}_{U}\to\mathrm{Hom}_{\mathbb{F}}(W,U)\big)$ consists of those elements $A$ of $\mathcal{W}_{U}$ (or of $\mathcal{P}_{U}$, or of $\mathrm{O}(V)$) that satisfy $A\big|_{U^{\perp}}=\mathrm{Id}_{U^{\perp}}$. Moreover, the previous paragraph shows that in this case, if $v \in V$ is such that $v+U^{\perp}$ corresponds to $v^{*} \in U^{*}$ then \[(v-Av)+U=-\psi^{*}(v^{*})=0\ (\mathrm{since\ }\psi^{*}=0)\mathrm{\ and\ hence\ }v-Av \in U,\] and this vector depends only on the value of $v$ in $V/U^{\perp}$. Identify $V/U^{\perp}$ with $U^{*}$ via Equation \eqref{isobil} once again, and denote the resulting map by $\eta:U^{*} \to U$. For $v$ and $w$ in $V$, such that $v+U^{\perp}$ and $w+U^{\perp}$ correspond to $v^{*}$ and $w^{*}$ in $U^{*}$ respectively, we recall that $Av-v$ and $Aw-w$ are in $U$, hence they are perpendicular to one another and their pairing with element of $V$ reduces to the pairing with images in $V/U^{\perp} \cong U^{*}$. We can thus write \[0=(Av,Aw)-(v,w)=(Av-v,w)+(v,Aw-w)=-(\eta v^{*},w^{*})-(v^{*},\eta w^{*})\] (by the definition of $\eta$), which shows that $\eta$ is indeed anti-symmetric. On the other hand, it is clear from the same calculation that given $\eta\in\mathrm{Hom}_{\mathbb{F}}^{as}(U^{*},U)$, the map sending $v \in V$ to $v-\eta v^{*}$, with $v^{*}$ corresponding to $v+U^{\perp}$ as above, is an element of $\mathcal{P}_{U}\subseteq\mathrm{O}(V)$ that lies in $\mathcal{W}_{U}$ and maps to the trivial element of $\mathrm{Hom}_{\mathbb{F}}(W,U)$. Hence $\ker\big(\mathcal{W}_{U}\to\mathrm{Hom}_{\mathbb{F}}(W,U)\big)$ is precisely $\mathrm{Hom}_{\mathbb{F}}^{as}(U^{*},U)$.

Let us now evaluate the commutator of two elements $A$ and $B$ of $\mathcal{W}_{U}$, with respective images $\psi$ and $\varphi$ in $\mathrm{Hom}_{\mathbb{F}}(W,U)$. For this we write, for $v \in V$ with $v^{*} \in U^{*}$ as above, the equality \[v-ABv=(v-Av)+(v-Bv)+(z-Az)\mathrm{\ for\ }z=Bv-v \in U^{\perp},\] in which we recall that $z+U=\varphi v^{*}$ by Equation \eqref{psidual} and therefore the summand $z-Az$ is $\psi(z+U)=\psi\varphi^{*}v^{*}$ by Equation \eqref{psidef}. A similar calculation, but now with $BAv$, produces the equality \[(BA-AB)v=(\psi\varphi^{*}-\varphi\psi^{*})v^{*}\mathrm{\ for\ every\ }v \in V\mathrm{\ with\ }v^{*}\mathrm{\ as\ above}.\] Now, as $v^{*}$ depends only on $v$ modulo $U^{\perp}$, and the difference $A^{-1}B^{-1}v-v$ lies in $U^{\perp}$, we can replace $v$ by $A^{-1}B^{-1}v$ in the left hand side of the last equality, which produces $v-ABA^{-1}B^{-1}v$. But by definition, this is the image of $v^{*}$ under the element of $\mathrm{Hom}_{\mathbb{F}}^{as}(U^{*},U)$ that is associated with the commutator $ABA^{-1}B^{-1}$, and the right hand side is twice the asserted anti-symmetric map. An application of Proposition \ref{Heischar} now completes the proof of the proposition.
\end{proof}

\smallskip

For determining the precise structure of $\mathcal{P}_{U}$, as well as of groups over $\mathbb{Z}$ below, we shall need coordinates for this group, which we define by what will essentially become an embedding of the Levi quotient from Lemma \ref{unipker} as a subgroup of $\mathcal{P}_{U}$. Such an embedding corresponds to the intersection with a complementary parabolic group of the same type, or equivalently by taking a complement $\tilde{U}$ for $U^{\perp}$ in $V$. Note that if $\tilde{U}$ is isotropic then this gives us a presentation of the Levi subgroup as the intersection of $\mathcal{P}_{U}$ with the parabolic subgroup $\mathcal{P}_{\tilde{U}}$. However, with the applications for working over $\mathbb{Z}$ in mind, where the (dual) lattice need not necessarily contain appropriate isotropic complements, we shall not assume that $\tilde{U}$ is isotropic.

Now, the fact that $\tilde{U}$ is a complement for $U^{\perp}$ in $V$ implies that each vector in $V$ can be altered by a unique element of $\tilde{U}$ to land in $U^{\perp}$, which via Equation \eqref{isobil} and the definition of $W$ yields an isomorphism
\begin{equation}
\tilde{U} \cong V/U^{\perp} \cong U^{*} \quad\mathrm{and\ hence}\quad U\cong\tilde{U}^{*},\quad\mathrm{as\ well\ as}\quad\tilde{W}:=(U+\tilde{U})^{\perp} \cong W. \label{Wtilde}
\end{equation}
Composition with the isomorphisms from Equation \eqref{Wtilde} gives identifications \[\mathrm{Hom}_{\mathbb{F}}(W,U)\stackrel{\sim}{\to}\mathrm{Hom}_{\mathbb{F}}(\tilde{W},U),\ \mathrm{O}(W)\stackrel{\sim}{\to}\mathrm{O}(\tilde{W}),\ \mathrm{Hom}_{\mathbb{F}}(U^{*},U)\stackrel{\sim}{\to}\mathrm{Hom}_{\mathbb{F}}(\tilde{U},U),\] \[\mathrm{Hom}_{\mathbb{F}}(U^{*},W)\stackrel{\sim}{\to}\mathrm{Hom}_{\mathbb{F}}(\tilde{U},\tilde{W}),\quad\mathrm{and}\quad\mathrm{End}_{\mathbb{F}}(U^{*})\stackrel{\sim}{\to}\mathrm{End}_{\mathbb{F}}(\tilde{U}),\] which we denote by
\begin{equation}
\psi\mapsto\tilde{\psi},\quad\gamma\mapsto\tilde{\gamma},\quad\eta\mapsto\tilde{\eta},\quad\psi^{*}\mapsto\widetilde{\psi^{*}},\quad\mathrm{and}\quad M^{*}\mapsto\widetilde{M^{*}} \label{tildeiso}
\end{equation}
respectively, for elements $\psi\in\mathrm{Hom}_{\mathbb{F}}(W,U)$, $\gamma\in\mathrm{O}(W)$, $\eta\in\mathrm{Hom}_{\mathbb{F}}(U^{*},U)$, and $M\in\mathrm{End}_{\mathbb{F}}(U)$, with $\psi^{*}\in\mathrm{Hom}_{\mathbb{F}}(U^{*},W)$ and $M^{*}\in\mathrm{End}_{\mathbb{F}}(U^{*})$. It is clear that $\eta$ is symmetric or anti-symmetric if and only if $\tilde{\eta}$ has that property.

The fact that $\tilde{U}$ may be non-isotropic requires us to introduce the following canonical correction homomorphism.
\begin{lem}
There exists a unique symmetric map $\alpha:\tilde{U} \to U$ such that the space $\{\tilde{u}-\alpha\tilde{u}|\;\tilde{u}\in\tilde{U}\}$ is isotropic. \label{nonisocor}
\end{lem}

\begin{proof}
Since $U$ is isotropic, we find that \[(\tilde{u}-\beta\tilde{u},\tilde{v}-\beta\tilde{v})=(\tilde{u},\tilde{v})-(\beta\tilde{u},\tilde{v})-(\tilde{u},\beta\tilde{v})\quad\mathrm{for\ every}\quad\beta:\tilde{U} \to U.\] Hence we are interested in homomorphisms $\beta\in\mathrm{Hom}_{\mathbb{F}}(\tilde{U},U)$ that satisfy \[(\beta\tilde{u},\tilde{v})+(\tilde{u},\beta\tilde{v})=(\tilde{u},\tilde{v})\quad\mathrm{for\ every\ }\tilde{u}\mathrm{\ and\ }\tilde{v}\mathrm{\ in\ }\tilde{U}.\] Now, the isomorphism $U\cong\tilde{U}^{*}$ from Equation \eqref{Wtilde} and the fact that $\mathbb{F}$ is not of characteristic 2 produce a map \[\alpha:\tilde{U} \to U\quad\mathrm{such\ that}\quad(\tilde{u},\tilde{v})=2(\alpha\tilde{u},\tilde{v})\quad\mathrm{for\ }\tilde{u}\mathrm{\ and\ }\tilde{v}\mathrm{\ in\ }\tilde{U},\] and it is clear that $\alpha$ is symmetric and satisfies the desired equality. Finally, if $\beta:\tilde{U} \to U$ also satisfies this equality then we get \[\beta=\alpha+(\beta-\alpha)\quad\mathrm{as\ well\ as}\quad\big((\beta-\alpha)\tilde{u},\tilde{v}\big)+\big(\tilde{u},(\beta-\alpha)\tilde{v}\big)=0\quad\mathrm{for\ every\ }\tilde{u}\mathrm{\ and\ }\tilde{v},\] so that $\beta-\alpha$ is anti-symmetric. Hence $\beta$ is symmetric if and only if $\beta=\alpha$. This proves the lemma.
\end{proof}
The stabilizer of $U\oplus\tilde{U}$ in $\mathcal{P}_{U}$ (which stabilizes also $\tilde{W}$ from Equation \eqref{Wtilde}) contains a unique Levi subgroup of that parabolic group, which is the intersection of $\mathcal{P}_{U}$ with the parabolic group associated with the isotropic subspace from Lemma \ref{nonisocor}. The choice of $\tilde{U}$ also allows us to investigate $\mathcal{W}_{U}$ in more detail, which establishes the following result.
\begin{prop}
The group $\mathcal{W}_{U}$ is the full Heisenberg group from Proposition \ref{unipstruc}, and the map from Lemma \ref{unipker} is surjective and splits. In particular, $\mathcal{P}_{U}$ is isomorphic to the semi-direct product in which $\mathrm{GL}(U)\times\mathrm{O}(W)$ acts on the Heisenberg group $\mathcal{W}_{U}$ via $(M,\gamma):(\psi,\eta)\mapsto(M\psi\gamma^{-1},M\eta M^{*})$. \label{pargrpgen}
\end{prop}

\begin{proof}
Consider an element $A\in\mathcal{W}_{U}$, lying over to the map $\psi\in\mathrm{Hom}_{\mathbb{F}}(W,U)$ as in Proposition \ref{unipstruc}, with the associated map $\tilde{\psi}\in\mathrm{Hom}_{\mathbb{F}}(\tilde{W},U)$ from Equation \eqref{tildeiso}. The proof of that proposition shows that \[A\big|_{U}=\mathrm{Id}_{U},\quad\mathrm{that\ for}\quad w\in\tilde{W} \cong W \quad\mathrm{we\ have}\quad Aw=w-\tilde{\psi}w,\] and that there is a map $\xi\in\mathrm{Hom}_{\mathbb{F}}(U^{*},U)$, which comes with the associated map $\tilde{\xi}\in\mathrm{Hom}_{\mathbb{F}}(\tilde{U},U)$ as in Equation \eqref{tildeiso}, such that \[A\tilde{u}=\tilde{u}+\widetilde{\psi^{*}}\tilde{u}-\tilde{\xi}\tilde{u}\quad\mathrm{for\ any\ }\tilde{u}\in\tilde{U},\quad\mathrm{with}\quad\widetilde{\psi^{*}}\tilde{u}\in\tilde{W}\quad\mathrm{and}\quad\tilde{\xi}\tilde{u} \in U.\] Recalling that $\tilde{W}$ is perpendicular to both $U$ and $\tilde{U}$ and that $U$ is isotropic, the fact that $A$ is an isometry implies that
\[0=(\tilde{u},\tilde{v})-(A\tilde{u},A\tilde{v})=(\tilde{\xi}\tilde{u},\tilde{v})+(\tilde{u},\tilde{\xi}\tilde{v}\big)-\big(\widetilde{\psi^{*}}\tilde{u},\widetilde{\psi^{*}}\tilde{v}\big)\quad\mathrm{for\ }\tilde{u}\mathrm{\ and\ }\tilde{v}\mathrm{\ in\ }\tilde{U},\] so that after applying the isomorphism $\tilde{U} \cong U^{*}$ from Equation \eqref{Wtilde} and inverting the maps from Equation \eqref{tildeiso} we get \[(\xi u^{*},v^{*})+(u^{*},\xi v^{*}\big)=(\psi^{*}u^{*},\psi^{*}v^{*})=(\psi\psi^{*}u^{*},v^{*})\quad\mathrm{for\ every\ }u^{*}\mathrm{\ and\ }v^{*}\mathrm{\ in\ }U^{*}.\] It follows that $\xi$ is the sum of the symmetric map $\frac{\psi\psi^{*}}{2}$ and an anti-symmetric map from $U^{*}$ to $U$. On the other hand, the same calculation shows that given any $\psi\in\mathrm{Hom}_{\mathbb{F}}(W,U)$, one can take $\xi$ to be $\frac{\psi\psi^{*}}{2}$ plus an arbitrary element of $\mathrm{Hom}_{\mathbb{F}}^{as}(U^{*},U)$, and the map $A$ defined by these formulae will be in $\mathrm{O}(V)$, hence in $\mathcal{P}_{U}$ and in $\mathcal{W}_{U}$, and will lie over $\psi$. This proves the first assertion.

For the second one, choose a pair of element $M\in\mathrm{GL}(U)$ and $\gamma\in\mathrm{O}(W)$, and we wish to construct an element $A\in\mathcal{P}_{U}$, that stabilizes $U\oplus\tilde{U}$ and $\tilde{W}$, and whose image in $\mathrm{GL}(U)\times\in\mathrm{O}(W)$ is $(M,\gamma)$. It is clear that we must take $A\big|_{U}=M$ and $A\big|_{\tilde{W}}\gamma$, and using the map $\alpha$ from Lemma \ref{nonisocor} we set \[A\tilde{u}=\widetilde{M^{-*}}\tilde{u}+M\alpha\tilde{u}-\alpha\widetilde{M^{-*}}\tilde{u} \in U\oplus\tilde{U}\quad\mathrm{for}\quad\tilde{u}\in\tilde{U},\quad\mathrm{where}\quad\widetilde{M^{-*}}=(\widetilde{M^{*}})^{-1}\] (i.e., $A$ takes the element $\tilde{u}-\alpha\tilde{u}$ of the isotropic space from that lemma to its natural $\widetilde{M^{-*}}$-image $\widetilde{M^{-*}}\tilde{u}-\alpha\widetilde{M^{-*}}\tilde{u}$ in that space, to which we must add the image $A\alpha\tilde{u}=M\alpha\tilde{u}$ of $\alpha\tilde{u} \in U$ under $A$). The perpendicularity of $\tilde{W}$ and $U\oplus\tilde{U}$ from Equation \eqref{Wtilde}, the orthogonality of $\gamma$, and the isotropy of $U$ easily imply that for verifying that the element $A$ thus defined lies in $\mathrm{O}(V)$, only the pairings of the form $(A\tilde{u},A\tilde{v})$ for $\tilde{u}$ and $\tilde{v}$ in $\tilde{U}$ have to be checked. But these observations combine with the proof of Lemma \ref{nonisocor} and the isomorphism between $\tilde{U}$ and $U^{*}$ given in Equation \eqref{Wtilde} to show that this pairing is the sum of \[\big(\widetilde{M^{-*}}\tilde{u},\widetilde{M^{-*}}\tilde{v}\big),\ \ \big(\widetilde{M^{-*}}\tilde{u},M\alpha\tilde{v}\big)=\tfrac{(\tilde{u},\tilde{v})}{2},\ \ -\big(\widetilde{M^{-*}}\tilde{u},\alpha\widetilde{M^{-*}}\tilde{v}\big)=-\tfrac{(\widetilde{M^{-*}}\tilde{u},\widetilde{M^{-*}}\tilde{v})}{2},\] \[\big(M\alpha\tilde{u},\widetilde{M^{-*}}\tilde{v}\big)=\tfrac{(\tilde{u},\tilde{v})}{2},\quad\mathrm{and}\quad -\big(\alpha\widetilde{M^{-*}}\tilde{u},\widetilde{M^{-*}}\tilde{v}\big)=-\tfrac{(\widetilde{M^{-*}}\tilde{u},\widetilde{M^{-*}}\tilde{v})}{2},\] which equals $(\tilde{u},\tilde{v})$ as desired. It is now immediate that $A$ lies in $\mathcal{P}_{U}$, stabilizes $U\oplus\tilde{U}$ and $\tilde{W}$, and projects to $(M,\gamma)$. The map sending $(M,\gamma)$ to $A$ thus defined is fairly easily verified to be an (injective) homomorphism of groups, whose image is the required Levi subgroup. This determines $\mathcal{P}_{U}$ as a semi-direct product, and a direct evaluation of the action by conjugation gives, using the fact that $\gamma\in\mathrm{O}(W)$, the asserted formula. This proves the proposition.
\end{proof}

It follows from Proposition \ref{pargrpgen} that the choice of $\tilde{U}$ defines a bijection of sets \[A\in\mathcal{P}_{U}\longleftrightarrow(M,\gamma,\psi,\eta)\in\mathrm{GL}(U)\times\mathrm{O}(W)\times\mathrm{Hom}_{\mathbb{F}}(W,U)\times\mathrm{Hom}_{\mathbb{F}}^{as}(U^{*},U),\] where given $u \in U$, $w\in\tilde{W}$, and $\tilde{u}\in\tilde{U}$, the action of the element $A$ on the left hand side is given by \[Au=Mu \in U,\qquad Aw=(\tilde{\gamma}w)-(\tilde{\psi}\tilde{\gamma}w)\in\tilde{W} \oplus U,\] and
\begin{equation}
A\tilde{u}=(\widetilde{M^{-*}}\tilde{u})+(\widetilde{\psi^{*}}\widetilde{M^{-*}}\tilde{u})+
\Big(M\alpha\tilde{u}-\alpha\widetilde{M^{-*}}\tilde{u}-\tfrac{\tilde{\psi}\widetilde{\psi^{*}}\widetilde{M^{-*}}\tilde{u}}{2}-\tilde{\eta}\widetilde{M^{-*}}\tilde{u}\Big)\in\tilde{U}\oplus\tilde{W} \oplus U. \label{PUgenform}
\end{equation}
Note that we have used the convention in which the element of the Levi subgroup from the proof of Proposition \ref{pargrpgen} operates first, and the one from $\mathcal{W}_{U}$ acts later. In the convention with the opposite order we have to replace $\psi$ and $\eta$ in Equation \eqref{PUgenform} (and in the formula for $Aw$) by their images under the action of the Levi element, which is described in that proposition.

Considering Equation \eqref{PUgenform} in the natural quotients arising from $U$ yields the following consequence.
\begin{cor}
The action of $\mathcal{P}_{U}$ on the space $V/U^{\perp}$, which is isomorphic to $U^{*}$ via Equation \eqref{Wtilde}, is via the dual of the $GL(U)$-part of the quotient $\mathcal{P}_{U}/\mathcal{W}_{U}$ from Lemma \ref{unipker} (or of the Levi subgroup). Its action on $V/U$ is via the semi-direct product in which $\mathrm{GL}(U)\times\mathrm{O}(W)$ acts on $\mathrm{Hom}_{\mathbb{F}}(W,U)$ via $(M,\gamma):\psi \mapsto M\psi\gamma^{-1}$. \label{actquot}
\end{cor}
In particular, the kernel of the two actions of $\mathcal{P}_{U}$ appearing in Corollary \ref{actquot} are the semi-direct product in which $\mathrm{O}(W)$ acts on $\mathcal{W}_{U}$ and the part $\mathrm{Hom}_{\mathbb{F}}^{as}(U^{*},U)$ of $\mathcal{W}_{U}$ respectively. We shall henceforth denote the finer quotient, namely $\mathcal{P}_{U}/\mathrm{Hom}_{\mathbb{F}}^{as}(U^{*},U)$, by $\overline{\mathcal{P}}_{U}$. The stabilizer of the direct sum $U\oplus\tilde{U}$ is the full group in the coarser quotient, and coincides with the image of the Levi subgroup in the finer one.

\smallskip

It is standard to show that when $\dim U\geq2$ the anti-symmetric map from Proposition \ref{unipstruc} is non-degenerate. On the other hand, if $\dim U=1$ then it is very degenerate, since $\mathrm{Hom}_{\mathbb{F}}^{as}(U^{*},U)=0$ in this case (there are no anti-symmetric bilinear forms on a 1-dimensional space), and we then have $\overline{\mathcal{P}}_{U}=\mathcal{P}_{U}$. Moreover, since in this case a choice of generator $z$ for $U$ and Equation \eqref{isobil} yield isomorphisms
\begin{equation}
\mathrm{Hom}_{\mathbb{F}}(W,U) \cong W^{*} \cong W,\quad\mathrm{and}\quad\mathrm{GL}(U)=\mathbb{F}^{\times}\mathrm{\ by\ definition\ when\ }\dim U=1, \label{isodim1}
\end{equation}
this special case of Proposition \ref{pargrpgen} takes the following form.
\begin{cor}
The stabilizer of a 1-dimensional isotropic subspace $U$ of $V$ is the semi-direct product in which the product $\mathbb{F}^{\times}\times\mathrm{O}(W)$ operates on the additive group of $W$ via $(c,\gamma):w \to c\gamma w$. It operates faithfully on the finer quotient $V/U$ from Corollary \ref{actquot}, while the action on the 1-dimensional coarser quotient $V/U^{\perp}$ is only via the scalar part $\mathbb{F}^{\times}$. \label{stdim1}
\end{cor}
Note that replacing the choice of the generator of $U$ that we used for identifying $\mathrm{Hom}_{\mathbb{F}}(W,U)$ with $W$ in Corollary \ref{stdim1} corresponds to a scalar rescaling of $W$, which does not affect the form of the semi-direct product there.

We also get a small simplification of Proposition \ref{pargrpgen} in case $\dim U=2$, since the space $\mathrm{Hom}_{\mathbb{F}}^{as}(U^{*},U)$, which is $\bigwedge^{2}U$ in the notation of \cite{[L]}, has dimension 1 in this case. Choosing a basis for $U$, hence the dual basis for $U^{*}$, yields a generator for $\mathrm{Hom}_{\mathbb{F}}^{as}(U^{*},U)$, and we get isomorphisms
\begin{equation}
\mathrm{GL}(U)\cong\mathrm{GL}_{2}(\mathbb{F}),\ \ \mathrm{Hom}_{\mathbb{F}}^{as}(U^{*},U)\cong\mathbb{F},\mathrm{\ and\ }\mathrm{Hom}_{\mathbb{F}}(W,U) \cong W^{*} \times W^{*} \cong W \times W. \label{isodim2}
\end{equation}
One verifies that the map from the latter space to $\mathrm{Hom}_{\mathbb{F}}^{as}(U^{*},U)\cong\mathbb{F}$ is the anti-symmetrization of the pairing on $W$, and as $\mathrm{GL}(U)$ has a natural action on $\mathrm{Hom}_{\mathbb{F}}^{as}(U^{*},U)=\bigwedge^{2}U=\det U$, the isomorphisms from Equation \eqref{isodim2} expresses Proposition \ref{pargrpgen} as follows.
\begin{cor}
When $\dim U=2$ the Heisenberg group from Proposition \ref{unipstruc} is isomorphic to the group $\tilde{H}(W,\mathbb{F})$ from Definition \ref{Heisdef}. In $\mathcal{P}_{U}$ it is acted upon by $\mathrm{GL}_{2}(\mathbb{F})\times\mathcal{O}(W)$, where $\mathrm{GL}_{2}(\mathbb{F})$ acts on the part $W \times W$ of $\tilde{H}(W,\mathbb{F})$ as on ($W$-valued) length 2 vectors, and on $\mathbb{F}$ via the determinant. \label{stdim2}
\end{cor}
We remark that in any dimension of $U$, a choice of a basis would yield isomorphisms of $\mathrm{GL}(U)$ with $\mathrm{GL}_{\dim U}(\mathbb{F})$ and of $\mathrm{Hom}_{\mathbb{F}}(W,U)$ with $(W^{*})^{\dim U}$ and therefore with $W^{\dim U}$ as in Equation \eqref{isodim2}, but since $\mathrm{Hom}_{\mathbb{F}}^{as}(U^{*},U)=\bigwedge^{2}U$ will no longer be 1-dimensional, it will only be useful for us in case $\dim U=2$.

\smallskip

There are two natural maps from $\mathrm{O}(V)$: One is the determinant to $\{\pm1\}$, and the other one is the spinor norm into $\mathbb{F}^{\times}/(\mathbb{F}^{\times})^{2}$. Their restrictions to $\mathcal{P}_{U}$ are evaluated as follows.
\begin{prop}
Both the determinant and the spinor norm of elements of $\mathcal{W}_{U}$ are trivial, hence these maps factor through the quotient $\mathrm{GL}(U)\times\mathrm{O}(W)$ from Lemma \ref{unipker}. If an element of $\mathcal{P}_{U}$ maps to a pair $(M,\gamma)$ in that product, then its determinant is just $\det\gamma$, and its spinor norm is the product of the image of $\det M$ in $\mathbb{F}^{\times}/(\mathbb{F}^{\times})^{2}$ with the spinor norm of $\gamma$. \label{detspnorm}
\end{prop}

\begin{proof}
The triviality of both maps on $\mathcal{W}_{U}$ is easily verified (as a unipotent group), and for the rest we can identify the quotient with the Levi subgroup from Proposition \ref{pargrpgen} then up to the correction factor involving the map $\alpha$ from Lemma \ref{nonisocor}. Since an element of that Levi subgroup operates as $M$ on $U$, as $\gamma$ on $\tilde{W}$, and as $\widetilde{M^{-*}}$ on $\tilde{U}^{*}$, and the latter is equivalent, via Equations \eqref{Wtilde} and \eqref{tildeiso}, to the action of $M^{-*}$ on $U^{*}$, the value of the determinant immediately follows. For the spinor norm it remains to consider the action on $U\oplus\tilde{U}$, and we note that if $U=\mathbb{F}u$ is 1-dimensional and $\tilde{u}\in\tilde{U}$ satisfies $(\tilde{u},u)=1$ then the map acting on $U$ as the scalar $c$ and on the isotropic counterpart of $\tilde{U}$ from Lemma \ref{nonisocor} as $\frac{1}{c}$ is the composition of the reflection in a vector $\tilde{u}-\alpha\tilde{u}-u$ composed with the reflection in $\tilde{u}-\alpha\tilde{u}-cu$. This gives the asserted spinor norm in the 1-dimensional case, and since extensions of maps of that sort from subspaces of dimension 1 generates $\mathrm{GL}(U)$ for $U$ of any dimension, this proves the proposition.
\end{proof}

When $\mathbb{F}$ is a subfield of $\mathbb{R}$ and we extend scalars to $\mathbb{R}$ if necessary, the spinor norm also becomes $\{\pm1\}$-valued. Then for indefinite $V$, of some signature $(b_{+},b_{-})$, the group $\mathrm{O}(V)$ has four connected components, the kernel $\mathrm{SO}(V)$ of the determinant consists of two of them, and the identity component is denoted by $\mathrm{SO}^{+}(V)$. When $V$ is definite we have $\mathrm{SO}(V)=\mathrm{SO}^{+}(V)$, and $\mathrm{O}(V)$ has only two connected components. Recall that in this case the dimension of any isotropic subspace $U$ of $V$ lies between 0 and $\min\{b_{+},b_{-}\}$, and when $\dim U=0$ we have $\mathcal{P}_{U}=\mathrm{O}(V)$ (over any field). Since $\mathrm{GL}(U)$ has two connected components wherever $U$ is non-trivial, the surjectivity of the map from Lemma \ref{unipker} (proved in Proposition \ref{pargrpgen}) combines with Proposition \ref{detspnorm} and the connectedness of $\mathcal{W}_{U}$ to show that when $U$ is non-trivial and $W=U^{\perp}/U$, of signature $(b_{+}-\dim U,b_{-}-\dim U)$, is indefinite (i.e., when $0<\dim U<\min\{b_{+},b_{-}\}$), the group $\mathcal{P}_{U}$ has eight connected components. For non-trivial $U$ of maximal dimension $\min\{b_{+},b_{-}\}$ (hence definite $W$), this group has four connected components when $W$ is non-trivial (i.e., when $b_{+} \neq b_{-}$), and only two in case $\dim U=b_{+}=b_{-}>0$ and $W=\{0\}$. The intersection $\mathcal{P}_{U}\cap\mathrm{SO}(V)$ always contains half of these connected components, and the connected components of $\mathcal{P}_{U}\cap\mathrm{SO}^{+}(V)$ as well as of the intersection of $\mathcal{P}_{U}$ with the kernel $\mathrm{O}^{+}(V)$ of the spinor norm on $\mathrm{O}(V)$ alone can be easily determined (the latter will depend on the signature of $W$ when this space is definite and non-trivial).

We also note that Corollary \ref{stdim1} with $\mathbb{F}=\mathbb{R}$ reproduces, when combined with Proposition \ref{detspnorm}, the result from \cite{[Nak]} and \cite{[Z2]} that the connected component of the stabilizer of the (oriented) line $\ell$, which is the stabilizer in $\mathrm{SO}^{+}(V)$ of a ray in $\ell$, is the semi-direct product in which $\mathbb{R}_{+}^{\times}\times\mathrm{SO}^{+}(W)$ operates on the additive group of the vector space $W=\ell^{\perp}/\ell$. If $W$ is indefinite then the stabilizer of $\ell$ in $\mathrm{SO}^{+}(V)$ contains also elements inverting the orientation of $\ell$, provided that the corresponding element of $\mathrm{O}(W)$ (or of $\mathrm{SO}(W)$) does \emph{not} lie in $\mathrm{SO}^{+}(W)$, in correspondence with Proposition \ref{detspnorm} (examples for such elements are those denoted by $k_{a,A}$ in \cite{[Nak]} and in \cite{[Z2]}, with $a<0$).

\section{Lattices \label{Lattices}}

The main goal of this section is to establish a decomposition of an even lattice in a non-degenerate rational quadratic space $V$, as well as of the dual lattice, in a way that is adapted to the isotropic subspace $U$ of $V$. This decomposition, given in Theorem \ref{latdecom}, is crucial for determining the intersection of the corresponding arithmetic subgroup with $\mathcal{P}_{U}$. Some cases are known, e.g., the analysis involving the lattice $L_{0}'$ on page 41 of \cite{[Br]} for the case where $\dim U=1$, but I am not aware of any reference dealing with the general case.

\medskip

From now on we consider the case $\mathbb{F}=\mathbb{Q}$, and we take $L$ to be an even lattice in $V$. This means that $L$ is a finitely generated subgroup of full rank in $V$, with \[(\lambda,\lambda)\in2\mathbb{Z}\mathrm{\ when\ }\lambda \in L,\mathrm{\ hence\ }(\lambda,\mu)\in\mathbb{Z}\mathrm{\ for\ }\lambda\mathrm{\ and\ }\mu\mathrm{\ in\ }L,\mathrm{\ and\ }V=L_{\mathbb{Q}}.\] The notation for duals for lattices will mean the $\mathbb{Z}$-dual, and in particular \[L^{*}:=\mathrm{Hom}(L,\mathbb{Z}) \subseteq V^{*}\quad\mathrm{is\ identified\ with}\quad\big\{\nu \in V\big|\;(\nu,L)\subseteq\mathbb{Z}\big\} \subseteq V.\] We shall henceforth denote this subgroup of $V$ by $L^{*}$ as well, and we have \[L \subseteq L^{*},\quad\mathrm{where\ the\ \emph{discriminant\ group}}\quad\Delta_{L}:=L^{*}/L \quad\mathrm{of\ }L\mathrm{\ is\ finite}.\] Since $L$ is even, the quadratic form $\lambda\mapsto\frac{\lambda^{2}}{2}$ yields a $\mathbb{Q}/\mathbb{Z}$-valued quadratic form on $\Delta_{L}$, which we shall also denote by $\mu\mapsto\frac{\mu^{2}}{2}$, in addition to the natural $\mathbb{Q}/\mathbb{Z}$-valued bilinear form, which is non-degenerate and will be denoted by $(\mu,\nu)$ as well. Since elements of $\mathrm{O}(V)$ (or of $\mathrm{O}(V_{\mathbb{R}})$) that preserve $L$ must also preserve $L^{*}$ hence act on $\Delta_{L}$ (preserving the $\mathbb{Q}/\mathbb{Z}$-quadratic structure), we obtain a map \[\mathrm{Aut}(L)\to\mathrm{Aut}(\Delta_{L}),\quad\mathrm{and\ we\ set}\quad\Gamma_{L}:=\ker\big(\mathrm{Aut}(L)\to\mathrm{Aut}(\Delta_{L})\big)\cap\mathrm{SO}^{+}(V_{\mathbb{R}}).\] The arithmetic subgroup $\Gamma_{L}$ has better integral properties than $\mathrm{Aut}(L)$ itself (in particular it is more functorial and it is more adapted to theta lifts).

The isotropic subspaces of $V$ are in one-to-one correspondence with primitive isotropic sublattices of $L$, under the natural inverse maps \[(U \subseteq V\mathrm{\ isotropic}) \mapsto I=U \cap L\quad\mathrm{and}\quad(I \subseteq L\mathrm{\ primitive\ isotropic}) \mapsto U=I_{\mathbb{Q}}.\] For $I$ and $U$ related in this way, we set \[\mathcal{P}_{U_{\mathbb{R}}}^{0}\mathrm{\ to\ be\ the\ identity\ component\ of\ }\mathcal{P}_{U_{\mathbb{R}}},\quad\mathrm{and}\quad\Gamma_{L,I}:=\mathcal{P}_{U_{\mathbb{R}}}^{0}\cap\Gamma_{L}\] (this is the same as $\mathcal{P}_{U}^{0}\cap\Gamma_{L}$ for $\mathcal{P}_{U}^{0}:=\mathcal{P}_{U}\cap\mathcal{P}_{U_{\mathbb{R}}}^{0}$, and it typically has index 2 inside $\mathcal{P}_{U}\cap\Gamma_{L}$). For analyzing it we shall require the following notion, in which we recall the difference in meaning between $X^{*}$ for a vector space $X$ and $\Lambda^{*}$ for a lattice $\Lambda$.
\begin{defn}
Let $X$ and $\tilde{X}$ be subspaces of $V$ on which the restriction of the pairing yields a non-degenerate bilinear form on $X\times\tilde{X}$, so that it identifies $\tilde{X}$ with $X^{*}$ (and equivalently $X$ with $\tilde{X}^{*}$), and let $\Lambda \subseteq X$ and $\tilde{\Lambda}\subseteq\tilde{X}$ be lattices. We say that the pairing between $\Lambda$ and $\tilde{\Lambda}$ is \emph{unimodular} if $(\lambda,\tilde{\lambda})\in\mathbb{Z}$ for every $\lambda\in\Lambda$ and $\tilde{\lambda}\in\tilde{\Lambda}$, and if the resulting map from $\tilde{\Lambda}$ to $\Lambda^{*}$, or equivalently from $\Lambda$ to $\tilde{\Lambda}^{*}$, is an isomorphism of Abelian groups. \label{unimod}
\end{defn}
In particular, the pairing between $L$ and $L^{*} \subseteq V$ is unimodular.

Given isotropic $U=I_{\mathbb{Q}} \subseteq V$ and $I=U \cap L \subseteq L$ as above, we denote \[I_{L^{*}}:=U \cap L^{*},\quad I^{\perp}_{L}:=I^{\perp} \cap L=U^{\perp} \cap L,\quad\mathrm{and}\quad I^{\perp}_{L^{*}}:=I^{\perp} \cap L^{*}=U^{\perp} \cap L^{*},\] where $I \subseteq I_{L^{*}}$ and $I^{\perp}_{L} \subseteq I^{\perp}_{L^{*}}$ with finite indices. The fact that $I$ and $I^{\perp}_{L}$ are primitive in $L$ and $I_{L^{*}}$ and $I^{\perp}_{L^{*}}$ are primitive in $L^{*}$ (by definition) and the unimodularity of the pairing between $L$ and $L^{*}$ imply the natural identifications \[I^{*} \cong L^{*}/I^{\perp}_{L^{*}},\quad(I_{L^{*}})^{*} \cong L/I^{\perp}_{L},\quad (I^{\perp}_{L})^{*} \cong L^{*}/I_{L^{*}},\mathrm{\ and\ }(I^{\perp}_{L^{*}})^{*} \cong L/I\quad\mathrm{over\ }\mathbb{Z}.\] It follows that \[\Lambda:=I^{\perp}_{L}/I \subseteq W=U^{\perp}/U\mathrm{\ is\ an\ even\ lattice,\ and\ }\Lambda^{*} \subseteq W^{*}=W\mathrm{\ is\ given\ by}\] \[\Lambda^{*}=\big\{\xi:I_{L}^{\perp}\to\mathbb{Z}\big|\;(\xi,I)=0\big\}=(I_{L}^{\perp})^{*} \cap I^{\perp}=(L^{*}/I_{L^{*}}) \cap I^{\perp}=I^{\perp}_{L^{*}}/I_{L^{*}}.\] As expected we get
\[\Lambda\subseteq\Lambda^{*}\mathrm{\ with\ finite\ index,\ and\ we\ set}\quad\Delta_{\Lambda}:=\Lambda^{*}/\Lambda\quad\mathrm{and}\quad p:\Lambda^{*}\to\Delta_{\Lambda}.\] It will also be useful to consider
\begin{equation}
L^{*}_{I}=\big\{\mu \in L^{*}\big|\;\exists\nu \in L,\ \forall\lambda \in I,\ (\mu,\lambda)=(\nu,\lambda)\big\}=L+I^{\perp}_{L^{*}} \subseteq L^{*}, \label{L*I}
\end{equation}
from which we deduce that \[L^{*}/L^{*}_{I}\cong^{\displaystyle{L^{*}/I^{\perp}_{L^{*}}}}\!\!\Big/\!_{\displaystyle{L^{*}_{I}/I^{\perp}_{L^{*}}}}=^{\displaystyle{L^{*}/I^{\perp}_{L^{*}}}}\!\!\Big/
\!_{\displaystyle{(L+I^{\perp}_{L^{*}})/I^{\perp}_{L^{*}}}}\cong^{\displaystyle{L^{*}/I^{\perp}_{L^{*}}}}\!\!\Big/\!_{\displaystyle{L/I^{\perp}_{L}}} \cong I^{*}/(I_{L^{*}})^{*}\!.\] To $I$ one associates the natural isotropic subgroup \[\quad H_{I}:=(L+I_{L^{*}})/L\subseteq\Delta_{L},\mathrm{\ which\ satisfies\ }H_{I} \cong I_{L^{*}}/I\mathrm{\ and\ thus\ }L^{*}/L^{*}_{I} \cong H_{I}^{*}.\] The perpendicularity of $I_{L^{*}}$ and $I^{\perp}_{L^{*}}$ now shows that $I^{\perp}_{L^{*}}/I^{\perp}_{L}=L^{*}_{I}/L$ is $\mathbb{Q}/\mathbb{Z}$-perpendicular to $H_{I}$ in $\Delta_{L}$, and since the last equality shows that it has index $|H_{I}|$ there, it must equal $H_{I}^{\perp}$ by non-degeneracy. Since one can also express $\Delta_{\Lambda}=\Lambda^{*}/\Lambda$ as \[^{\displaystyle{I^{\perp}_{L^{*}}/I_{L^{*}}}}\Big/_{\displaystyle{I^{\perp}_{L}/I}} \cong I^{\perp}_{L^{*}}/(I^{\perp}_{L}+I_{L^{*}})\cong^{\displaystyle{I^{\perp}_{L^{*}}/I^{\perp}_{L}}}\Big/_{\displaystyle{(I^{\perp}_{L}+I_{L^{*}})/I^{\perp}_{L}}}
\cong^{\displaystyle{I^{\perp}_{L^{*}}/I^{\perp}_{L}}}\Big/_{\displaystyle{I_{L^{*}}/I}},\] our formulae for $H_{I}$ and $H_{I}^{\perp}$ identify $\Delta_{\Lambda}$ with $H_{I}^{\perp}/H_{I}$.

\smallskip

For giving good coordinates for $\mathcal{P}_{U}$ (as in, e.g., Equation \eqref{PUgenform}), we required a complementary subspace $\tilde{U}$ for $U^{\perp}$ in $V$. Here we shall need it to be complementary over $\mathbb{Z}$, as defined in the following lemma.
\begin{lem}
Let $U$ be an isotropic subspace of $V$, and set $I=U \cap L$. Then there exists a sublattice $\tilde{I}$ of $L^{*}$ whose pairing with $I$ is unimodular in the sense of Definition \ref{unimod}. Such a sublattice is primitive in $L^{*}$. \label{compoverZ}
\end{lem}

\begin{proof}
The fact that $I$ is primitive in $L$ means that $L=I \oplus J$ for some subgroup $J$ of $L$ (not necessarily orthogonal to $I$), implying that $L^{*}=\mathrm{Hom}(L,\mathbb{Z})$ is isomorphic to $I^{*} \oplus J^{*}$. Considering $L^{*}$ as a subgroup of $V$, the part corresponding to $I^{*}$ becomes the required sublattice $\tilde{I}$, which is clearly primitive in $L^{*}$. This proves the lemma.
\end{proof}
We shall use only complements $\tilde{U}$ that are of the form $\tilde{I}_{\mathbb{Q}}$ for a sublattice $\tilde{I}$ of $L^{*}$ satisfying the condition from Lemma \ref{compoverZ} (this is the reason why we cannot always take $\tilde{U}$ to be isotropic, since isotropic $\tilde{I}$ with this property need not always exist). Recalling $L^{*}_{I}$ from Equation \eqref{L*I}, we shall also denote
\begin{equation}
\tilde{I}_{L}=\big\{\mu\in\tilde{I}\big|\;\exists\nu \in L,\ \forall\lambda \in I,\ (\mu,\lambda)=(\nu,\lambda)\big\}=L^{*}_{I}\cap\tilde{I}, \label{tildeIL}
\end{equation}
which is a primitive sublattice of $L^{*}_{I}$. Note that $\tilde{I}_{L}$ is the subgroup of $\tilde{I}$ that pairs in a unimodular manner with $I_{L^{*}}$, so that $(\tilde{I}_{L})^{*} \cong I_{L^{*}}$. However, we shall make use of the $\mathbb{Q}/\mathbb{Z}$-dual of $\tilde{I}_{L}$, which is \[\mathrm{Hom}(\tilde{I}_{L},\mathbb{Q}/\mathbb{Z})=\mathrm{Hom}(\tilde{I}_{L},\mathbb{Q})/\mathrm{Hom}(\tilde{I}_{L},\mathbb{Z})=\tilde{U}^{*}/(\tilde{I}_{L})^{*} \cong U/I_{L^{*}}\] because $(\tilde{I}_{L})_{\mathbb{Q}}=\tilde{I}_{\mathbb{Q}}=\tilde{U}$ and $\tilde{U}^{*} \cong U$ by Equation \eqref{Wtilde}. Recall that this equation also contains the isomorphism \[W\cong\tilde{W}=(U\oplus\tilde{U})^{\perp},\mathrm{\ and\ denote\ the\ images\ of\ }\Lambda\subseteq\Lambda^{*} \subseteq W\mathrm{\ by\ }\tilde{\Lambda}\subseteq\tilde{\Lambda}^{*}\subseteq\tilde{W}.\] It follows that \[\tilde{\Lambda}^{*}/\tilde{\Lambda}\cong\Lambda^{*}/\Lambda=\Delta_{\Lambda},\quad\mathrm{and\ we\ denote\ the\ projection\ }\tilde{\Lambda}^{*}\to\Delta_{\Lambda}\mathrm{\ by\ }\tilde{p}.\]

Analyzing $\Gamma_{L,I}$ will require the decompositions of $L$ and $L^{*}$ according to the splitting of $V$ as $U\oplus\tilde{W}\oplus\tilde{U}$. Our first main result does this also for $L^{*}_{I}$.
\begin{thm}
There exists a homomorphism
\[\iota:\tilde{I}_{L} \to \Delta_{\Lambda},\quad\mathrm{which\ satisfies}\quad\tfrac{(\iota\tilde{u})^{2}}{2}=\tfrac{\tilde{u}^{2}}{2}+\mathbb{Z}\in\mathbb{Q}/\mathbb{Z}\quad\mathrm{for\ every}\quad\tilde{u}\in\tilde{I}_{L},\] such that for $u \in U$, $w\in\tilde{W}$, and $\tilde{u}\in\tilde{U}$ the sum $u+w+\tilde{u}$ is in $L$ if and only if
\begin{equation}
\tilde{u}\in\tilde{I}_{L},\quad w\in\tilde{\Lambda}^{*}\mathrm{\ with\ }\tilde{p}w=w+\tilde{\Lambda}=\iota\tilde{u}\in\Delta_{\Lambda},\quad\mathrm{and}\quad u\in-2\alpha \tilde{u}+I, \label{decomL}
\end{equation}
where $\alpha$ is the map from Lemma \ref{nonisocor}. On the other hand, given $u$, $w$, and $\tilde{u}$ as above, the sum $u+w+\tilde{u}$ lies in $L^{*}$ (resp. in $L^{*}_{I}$) if and only if
\begin{equation}
w\in\tilde{\Lambda}^{*},\quad u+I_{L^{*}}=-\iota^{*}(w+\tilde{\Lambda})=-\iota^{*}\tilde{p}w,\quad\mathrm{and}\quad\tilde{u}\in\tilde{I}\ (\mathrm{resp.\ }\tilde{u}\in\tilde{I}_{L}), \label{decomL*}
\end{equation}
where $\iota^{*}:\Delta_{\Lambda} \to U/I_{L^{*}}$ is the $\mathbb{Q}/\mathbb{Z}$-dual of $\iota:\tilde{I}_{L} \to \Delta_{\Lambda}$. \label{latdecom}
\end{thm}

\begin{proof}
The unimodularity of the pairing of $\tilde{I}$ from Lemma \ref{compoverZ} with $I$, the definition of $L^{*}_{I}$ and $\tilde{I}_{L}$ in Equations \eqref{L*I} and \eqref{tildeIL} respectively, and the fact that $I^{\perp}_{L^{*}} \subseteq L^{*}_{I} \subseteq L^{*}$, combine to show that
\[L^{*}=I^{\perp}_{L^{*}}\oplus\tilde{I} \qquad\mathrm{and}\qquad L^{*}_{I}=I^{\perp}_{L^{*}}\oplus\tilde{I}_{L}.\] This reduces the proof of Equation \eqref{decomL*} to the determination of the decomposition of $I^{\perp}_{L^{*}}$ inside $U\oplus\tilde{W}$. On the other hand, we recall that \[0 \to I \to I^{\perp}_{L}\to\Lambda\to0\mathrm{\ is\ exact},\quad\mathrm{and\ that}\quad(I^{\perp}_{L},\tilde{I})\subseteq(L,\tilde{I})\subseteq\mathbb{Z}\mathrm{\ \ and\ \ }\tilde{I}^{*} \cong I,\] from which it follows via the definition of $\tilde{W}$ in Equation \eqref{Wtilde} that
\begin{equation}
L=I\oplus(L\cap\tilde{I}^{\perp})\quad\mathrm{and}\quad I^{\perp}_{L}=I\oplus(L\cap\tilde{W}),\quad\mathrm{and\ therefore}\quad I^{\perp}_{L}=I\oplus\tilde{\Lambda}. \label{IperpL}
\end{equation}
Since the pairing of $I^{\perp}_{L^{*}}$ with $I^{\perp}_{L}$ is integral, and an element of $U\oplus\tilde{W}$ lies in $U\oplus\tilde{\Lambda}^{*}$ if and only if it pairs integrally with $I\oplus\tilde{\Lambda}$, Equation \eqref{IperpL} implies that \[I^{\perp}_{L^{*}} \subseteq U\oplus\tilde{\Lambda}^{*}, \qquad\mathrm{and}\qquad I^{\perp}_{L^{*}}\cap(U\oplus\tilde{\Lambda})=(I^{\perp}_{L^{*}} \cap U)+I^{\perp}_{L}=I_{L^{*}}+I^{\perp}_{L}=I_{L^{*}}\oplus\tilde{\Lambda}.\] Recalling that $I^{\perp}_{L^{*}}$ projects onto $\Lambda^{*}$, and noting that the inverse image of $\Lambda$ in that projection is $I_{L^{*}}\oplus\tilde{\Lambda}$ by the last equation, we deduce the existence of a homomorphism $\hat{\iota}:\Delta_{\Lambda} \to U/I_{L^{*}}$ such that if $u \in U$ and $w\in\tilde{\Lambda}^{*}$ then
\begin{equation}
u+w \in I^{\perp}_{L^{*}}\quad\Longleftrightarrow\quad u+I_{L^{*}}=-\hat{\iota}\tilde{p}w \in U/I_{L^{*}},\quad\mathrm{with}\quad\tilde{p}w=w+\tilde{\Lambda}\in\Delta_{\Lambda}. \label{IperpLdual}
\end{equation}
The proof of Equation \eqref{decomL*} thus reduces to finding the homomorphism $\iota$ for Equation \eqref{decomL}, and showing that $\hat{\iota}=\iota^{*}$.

The definitions of $L^{*}_{I}$ and $\tilde{I}_{L}$ in Equations \eqref{L*I} and \eqref{tildeIL} and our description of $I^{\perp}_{L^{*}}$ now imply that \[L \subseteq L+I^{\perp}_{L^{*}}=L^{*}_{I}=I^{\perp}_{L^{*}}\oplus\tilde{I}_{L} \subseteq U\oplus\tilde{\Lambda}^{*}\oplus\tilde{I}_{L},\] so that the projection from $L$ to $\tilde{I}_{L}$ in these coordinates is surjective, and Equation \eqref{IperpL} shows that $L\cap(U\oplus\tilde{\Lambda}^{*})=I^{\perp}_{L}=I\oplus\tilde{\Lambda}$. It follows that there are homomorphisms $\iota:\tilde{I}_{L}\to\Delta_{\Lambda}$ and $\hat{\alpha}:\tilde{I}_{L} \to U/I$ such that for a triple $u \in U$, $w\in\tilde{\Lambda}^{*}$, and $\tilde{u}\in\tilde{U}$ we have $u+w+\tilde{u} \in L$ if and only if \[\tilde{u}\in\tilde{I}_{L},\quad w\in\tilde{\Lambda}^{*}\mathrm{\ with\ }\tilde{p}w=w+\tilde{\Lambda}=\iota\tilde{u}\in\Delta_{\Lambda},\quad\mathrm{and}\quad u+I=-\hat{\alpha}\tilde{u},\] which is almost Equation \eqref{decomL}. The inclusion $(L,I^{\perp}_{L^{*}})\subseteq\mathbb{Z}$, Equation \eqref{IperpLdual}, the isotropy of $U$, and the perpendicularity of $\tilde{W}$ and $\tilde{U}$ from Equation \eqref{Wtilde} now easily imply the equality $\hat{\iota}=\iota^{*}$, thus establishing Equation \eqref{decomL*}. For proving Equation \eqref{decomL} we only need to show that if $\tilde{u}\in\tilde{I}_{L}$ and $\alpha$ is the map from Lemma \ref{nonisocor} then $\hat{\alpha}\tilde{u}$ is the coset $2\alpha\tilde{u}+I$ in $U/I$. The decomposition of $L$ in Equation \eqref{IperpL} and the fact that we require the image of $u$ modulo $I$ allow us to restrict attention to $u+w+\tilde{u}\in L\cap\tilde{I}^{\perp}$, whose pairing with $\tilde{v}\in\tilde{U}$ vanishes. As $\tilde{v}\in\tilde{I}$ is perpendicular to $\tilde{W}$ by definition, we obtain from the proof of Lemma \ref{nonisocor} that \[(u+\tilde{u},\tilde{v})=0\mathrm{\ hence\ }(-u,\tilde{v})=(\tilde{u},\tilde{v})=2(\alpha\tilde{u},\tilde{v})\mathrm{\ for\ }\tilde{v}\in\tilde{U},\mathrm{\ and\ thus\ }u=-2\alpha\tilde{u}\] (since $u$ and $2\alpha\tilde{u}$ are in $U$ and $U\cong\tilde{U}^{*}$ via Equation \eqref{Wtilde}). This implies the desired relation $\hat{\alpha}\tilde{u}=-u+I=2\alpha\tilde{u}+I$ between $\hat{\alpha}$ and $\alpha$, and Equation \eqref{decomL} follows.

It only remains to prove the norm property of $\iota$. For this, take $u$, $w$, and $\tilde{u}$ as in Equation \eqref{decomL}, and recall that $L$ is an even lattice. Thus the expression \[\tfrac{(u+w+\tilde{u})^{2}}{2}=\tfrac{w^{2}}{2}+\tfrac{\tilde{u}^{2}}{2}+(\tilde{u},u)=\tfrac{w^{2}}{2}+\tfrac{\tilde{u}^{2}}{2}-2(\tilde{u},\alpha\tilde{u})+(\tilde{u},v)=
\tfrac{w^{2}}{2}-\tfrac{\tilde{u}^{2}}{2}+(\tilde{u},v),\] where we have used the fact that $u=-2\alpha\tilde{u}+v$ for some $v \in I$ and Lemma \ref{nonisocor}, must be integral, and since $(I,\tilde{I}_{L})\subseteq\mathbb{Z}$, we may ignore the last summand. The fact that when $w\in\tilde{\Lambda}^{*}$ the class of $\frac{w^{2}}{2}$ in $\mathbb{Q}/\mathbb{Z}$ depends only on $w+\tilde{\Lambda}\in\Delta_{\Lambda}$, and this coset is $\iota\tilde{u}$ by Equation \eqref{decomL}, thus proves the desired norm condition. This completes the proof of the theorem.
\end{proof}
Equation \eqref{decomL*} in Theorem \ref{latdecom} also reproduces the isomorphism between the groups $L^{*}/L^{*}_{I}$ and $I^{*}/(I_{L^{*}})^{*}$ (and with it the equality $[L^{*}:L^{*}_{I}]=|H_{I}|$), since we have seen that $\tilde{I} \cong I^{*}$ and $\tilde{I}_{L}\cong(I_{L^{*}})^{*}$.

Equations \eqref{decomL} and \eqref{decomL*} take the simplest form in case $\iota$ is the trivial map. Indeed, in this case they reduce to
\begin{equation}
L^{*}=I_{L^{*}}\oplus\tilde{\Lambda}^{*}\oplus\tilde{I},\qquad L^{*}_{I}=I_{L^{*}}\oplus\tilde{\Lambda}^{*}\oplus\tilde{I}_{L},\qquad\mathrm{and}\qquad L \subseteq U\oplus\tilde{\Lambda}\oplus\tilde{I}_{L} \label{simpdecom}
\end{equation}
with only the $\alpha$-corrections in the $U$-coordinates of elements of $L$ (which are therefore contained in $I_{L^{*}}$). In general, for a given choice of $U$ (and $I$) there are many possible choices for the complement $\tilde{U}$ (or $\tilde{I}$) satisfying the condition of Lemma \ref{compoverZ}. In order to compare the resulting maps from Theorem \ref{latdecom}, we forget the norm condition, and consider $\iota$ to be defined on $(I_{L^{*}})^{*}$ (whose definition does not depend on $\tilde{U}$), and since this group is contained in $I^{*}$ we have a restriction map \[\mathrm{Res}^{I^{*}}_{(I_{L^{*}})^{*}}:\mathrm{Hom}(I^{*},\Delta_{\Lambda})\to\mathrm{Hom}\big((I_{L^{*}})^{*},\Delta_{\Lambda}\big).\] The behavior of $\iota$ under changing the choice of $\tilde{U}$ is as follows.
\begin{prop}
Replacing $\tilde{U}$ and $\tilde{I}$ by another complement as in Lemma \ref{compoverZ} changes the map $\iota:(I_{L^{*}})^{*}\to\Delta_{\Lambda}$ from Theorem \ref{latdecom} by the restriction of a homomorphism from $I^{*}$ to $\Delta_{\Lambda}$. In particular, the choice of $U$ (and $I$) determines a class in $\mathrm{Hom}\big((I_{L^{*}})^{*},\Delta_{\Lambda}\big)/\mathrm{Res}^{I^{*}}_{(I_{L^{*}})^{*}}\big(\mathrm{Hom}(I^{*},\Delta_{\Lambda})\big)$. \label{choicecomp}
\end{prop}

\begin{proof}
Denote the new complementary lattice, which satisfies the condition of Lemma \ref{compoverZ} by definition, by $\hat{I}$, and set $\hat{U}=\hat{I}_{\mathbb{Q}}$. As $\hat{I} \subseteq L^{*}$ pairs with $I$ in a unimodular manner, decomposing $L^{*}$ as in Equation \eqref{decomL*} and projecting onto $\tilde{I}$ yields an isomorphism, whose inverse is a map $\tilde{I}\to\hat{I}$ that we write as $\tilde{u}\mapsto\hat{u}$ (i.e., given $\tilde{u}\in\tilde{I}$ we denote by $\hat{u}$ the unique element of $\hat{I}$ whose $\tilde{U}$-coordinate in Equation \eqref{decomL*} is $\tilde{u}$). We therefore have $\hat{I}=\{\hat{u}=(\beta\tilde{u},\tilde{\varphi}\tilde{u},\tilde{u})|\;\tilde{u}\in\tilde{U}\}$ in the coordinates of $V$ as $U\oplus\tilde{W}\oplus\tilde{U}$, for two homomorphisms
\begin{equation}
\varphi:\tilde{I}\to\Lambda^{*}\quad\mathrm{and}\quad\beta:\tilde{I} \to U,\quad\mathrm{with}\quad\beta\tilde{u}+I_{L^{*}}=-\iota^{*}p\varphi\tilde{u}\quad\mathrm{for\ every}\quad\tilde{u}\in\tilde{U}, \label{hatI}
\end{equation}
where $\tilde{\varphi}\tilde{u}$ is the image in $\tilde{\Lambda}^{*}$ of $\varphi\tilde{u}\in\Lambda^{*}$ as in Equation \eqref{tildeiso}. We extend $\varphi$ and $\tilde{\varphi}$ to maps \[\varphi:\tilde{U} \to W\mathrm{\ and\ }\tilde{\varphi}:\tilde{U}\to\tilde{W},\quad\mathrm{with\ duals}\quad\varphi^{*}:W \to U\mathrm{\ and\ }\tilde{\varphi}^{*}:\tilde{W} \to U,\] and recall that when using $\hat{I}$ and $\hat{U}$ and the corresponding Equation \eqref{Wtilde}, we consider \[\Lambda\subseteq\Lambda^{*} \subseteq W\mathrm{\ via\ their\ images\ }\hat{\Lambda}\subseteq\hat{\Lambda}^{*}\subseteq\hat{W}=(U+\hat{U})^{\perp},\mathrm{\ with\ }\hat{p}:\hat{\Lambda}^{*}\to\Delta_{\Lambda}.\] For determining the latter space, we take $w\in\tilde{W}$, and since \[(w,\hat{u})=\big(w,\tilde{\varphi}\tilde{u}\big)=\big(\tilde{\varphi}^{*}w,\tilde{u}\big)=\big(\tilde{\varphi}^{*}w,\hat{u}\big)\quad(\mathrm{as\ }\tilde{\varphi}^{*}w \in U\mathrm{\ and\ }\hat{u}-\tilde{u} \in U^{\perp}),\] we deduce that \[\hat{W}=\{\hat{w}=(-\tilde{\varphi}^{*}w,w,0)|\;w\in\tilde{W}\},\mathrm{\ \ as\ well\ as\ \ \ }\hat{\varphi}\tilde{u}=(-\tilde{\varphi}^{*}\tilde{\varphi}\tilde{u},\tilde{\varphi}\tilde{u},0)\mathrm{\ for\ }\tilde{u}\in\tilde{U}\] in the same coordinates, for the corresponding map $\hat{\varphi}:\tilde{U}\to\hat{W}$.

For the effect on $\iota$, consider a sum $u+w+\tilde{u} \in L$ as above, and write \[\tilde{u}=\hat{u}-\tilde{\varphi}\tilde{u}-\beta\tilde{u},\quad\mathrm{ as\ well\ as}\quad w=\hat{w}+\tilde{\varphi}^{*}w\mathrm{\ and\ }\tilde{\varphi}\tilde{u}=\hat{\varphi}\tilde{u}+\tilde{\varphi}^{*}\tilde{\varphi}\tilde{u}.\] This changes the class $\iota\tilde{u}=\tilde{p}w=\hat{p}\hat{w}$ to $\tilde{p}(w-\tilde{\varphi}\tilde{u})=\hat{p}(\hat{w}-\hat{\varphi}\hat{u})$ (where we make the abuse of notation of writing $\hat{\varphi}$ also for the map from $\hat{U}$ to $\hat{W}$, which takes the image $\hat{u}\in\hat{U}$ of $\tilde{u}\in\tilde{U}$ to $\hat{\varphi}\tilde{u}$), and therefore subtracts from $\iota$ the restriction to $\tilde{I}_{L}\cong(I_{L^{*}})^{*}$ of the composition $p\varphi$ on $\tilde{I} \cong I^{*}$. As for the condition on $u$ in Equation \eqref{decomL}, we express it via Lemma \ref{nonisocor} and the duality between $I$ and $\tilde{I}$ as the condition that
$(u,\tilde{v})\in-(\tilde{u},\tilde{v})+\mathbb{Z}$ for every $\tilde{v}\in\tilde{I}$, and considering the modified value of $u$ we have to prove that \[(u-\beta\tilde{u}+\tilde{\varphi}^{*}w-\tilde{\varphi}^{*}\tilde{\varphi}\tilde{u},\hat{v})\in-(\hat{u},\hat{v})+\mathbb{Z}\quad\mathrm{for\ every\ }\quad\hat{v}\in\hat{I}.\] We can replace $\hat{v}$ by $\tilde{v}$ on the left hand side (since their difference is in $U^{\perp}$), and using duality, what we know about $(u,\tilde{v})$, and the direct evaluation of $(\hat{u},\hat{v})$, this side gives an element of \[-(\tilde{u},\tilde{v})-(\beta\tilde{u},\tilde{v})+(w,\varphi\tilde{v})-(\varphi\tilde{u},\varphi\tilde{v})+\mathbb{Z}=-(\hat{u},\hat{v})+(w,\varphi\tilde{v})+(\tilde{u},\beta\tilde{v})+\mathbb{Z}.\] But $\varphi\tilde{v}\in\Lambda^{*}$, and we have $\tilde{p}w=\iota\tilde{u}$ by Equation \eqref{decomL}, so that the image of second term on the right hand side in $\mathbb{Q}/\mathbb{Z}$ is $(\iota\tilde{u},p\varphi\tilde{v})$. On the other hand, the fact that $\tilde{u}\in\tilde{I}_{L}$ implies that $(\tilde{u},I_{L^{*}})\subseteq\mathbb{Z}$, and when we take the image of the third term there in $\mathbb{Q}/\mathbb{Z}$ then the expression for $\beta\tilde{v}+I_{L^{*}}$ in Equation \eqref{hatI} transforms this element into $-(\tilde{u},\iota^{*}p\varphi\tilde{v})$. As this cancels with the previous expression, the right hand side is $-(\hat{u},\hat{v})+\mathbb{Z}$ as desired. Similar considerations show that the $\mathbb{Q}/\mathbb{Z}$-image of $\frac{\hat{u}^{2}}{2}$ coincides with $\frac{[(\iota-p\varphi)\tilde{u}]^{2}}{2}$ when $\hat{u}$ is associated with $\tilde{u}\in\tilde{I}_{L}$ as above, and as passing through the isomorphisms $\tilde{I}_{L}\cong(I_{L^{*}})^{*}\cong\hat{I}_{L}$ (with $\hat{I}_{L}$ defined as in Equation \eqref{tildeIL} with $\tilde{I}$ replaced by $\hat{I}$) allows us to write the latter expression as $\frac{[(\iota-p\varphi)\hat{u}]^{2}}{2}$, this verifies the norm condition from Theorem \ref{latdecom}. Note that from $\varphi(\tilde{I})\subseteq\Lambda^{*}$ we obtain $\varphi^{*}(\Lambda) \subseteq I$ by dualizing and hence $\tilde{\varphi}^{*}(\tilde{\Lambda}) \subseteq I$, so that the form of $I^{\perp}_{L}$ as $I\oplus\hat{\Lambda}$ as in Equation \eqref{IperpL} is also preserved. This establishes the first assertion, from which the second one directly follows. This proves the Proposition.
\end{proof}
Recall that $\Lambda^{*}$ and $I^{\perp}_{L^{*}}$ are torsion-free, and the projection $p:\Lambda^{*}\to\Delta_{\Lambda}$ and the map from $I^{\perp}_{L^{*}}$ to $\Lambda^{*}$ are surjective. Hence homomorphisms from $\tilde{I}$ to $\Delta_{\Lambda}$ can always be lifted to maps $\varphi:\tilde{I}\to\Lambda^{*}$, and then to maps from $\tilde{I}$ to $I^{\perp}_{L^{*}}$ with an appropriate homomorphism $\beta$ as in Equation \eqref{hatI}. Therefore no finer invariant can be associated to $U$ and $I$ themselves in Proposition \ref{choicecomp}. In particular, a necessary and sufficient condition for the existence of a convenient complement $\hat{U}$ for which the associated map $\iota$ will vanish, so that Equations \eqref{decomL} and \eqref{decomL*} will take the simpler form appearing in Equation \eqref{simpdecom}, is that the class from Proposition \ref{choicecomp} be trivial.

\section{Intersections with Arithmetic Subgroups \label{ArithSbgp}}

This section proves the conditions for elements of $\mathcal{P}_{U}$ to be in $\Gamma_{L,I}$, and deduces some consequences under simplifying assumptions. The main result of this paper is essentially Theorem \ref{sdprodZ}.

\medskip

First, the fact that the group $\Gamma_{L,I}$ was defined to be contained in $\mathcal{P}_{U}^{0}$ means that in Equation \eqref{PUgenform} we consider only elements $A\in\mathcal{P}_{U}$ that are represented by parameters $M\in\mathrm{GL}(U)$ with positive determinant, $\gamma\in\mathrm{O}(W)\cap\mathrm{SO}^{+}(W_{\mathbb{R}})$, $\psi\in\mathrm{Hom}_{\mathbb{Q}}(W,U)$, and $\eta\in\mathrm{Hom}_{\mathbb{Q}}^{as}(U^{*},U)$. In addition, we recall that if an element $M\in\mathrm{GL}(U)$ sends $I$ onto itself and $\det M>0$ then $M$ is in the group $\mathrm{SL}(I)$ of $\mathbb{Z}$-automorphisms of $I$ with determinant 1. Moreover, for any lattice $I \subseteq J \subseteq U$ (such as $J=I_{L^{*}}$) we define
\begin{equation}
\mathrm{SL}(J,I):=\big\{M\in\mathrm{SL}(I)\big|\;\forall u \in J, u-Mu \in I\big\}\subseteq\mathrm{SL}(I)\cap\mathrm{SL}(J). \label{SLJIdef}
\end{equation}
It follows that for every $M\in\mathrm{SL}(J,I)$, the map $\mathrm{Id}_{U}-M$ induces a well-defined map from $U/J$ to $U/I$. In addition, we have the group $\Gamma_{\Lambda}\subseteq\mathrm{SO}^{+}(W)$ defined in analogy to $\Gamma_{L}\subseteq\mathrm{SO}^{+}(V)$, and since our coordinates are already based on the choice of a complement $\tilde{U}=\tilde{I}_{\mathbb{Q}}$ as above, the map $\iota:\tilde{I}_{L}\to\Delta_{\Lambda}$ from Theorem \ref{latdecom}, as well as its dual $\iota^{*}:\Delta_{\Lambda} \to U/I_{L^{*}}$, can be considered as given.

For proving the condition characterizing $\Gamma_{L,I}$ we shall use the following technical lemma.
\begin{lem}
Let $B$ be a symmetric bilinear form on $\tilde{I}$, with values in $\mathbb{Q}$, which we view as a symmetric element of $\mathrm{Hom}(\tilde{I},U)=\mathrm{Hom}_{\mathbb{Q}}(\tilde{U},U)$ via the isomorphism between $\tilde{U} \cong U^{*}$ from Equation \eqref{Wtilde}. Then there exists $\tilde{\kappa}\in\mathrm{Hom}(\tilde{I},I)$ such that $B-\tilde{\kappa}$ is anti-symmetric if and only if $B(\tilde{u},\tilde{u})\in\mathbb{Z}$ for every $\tilde{u}\in\tilde{I}$, and in this situation the element $\tilde{\kappa}$, and with it the anti-symmetric map $B-\tilde{\kappa}$, is unique up to $\mathrm{Hom}^{as}(\tilde{I},I)$. \label{symas}
\end{lem}

\begin{proof}
If $B-\tilde{\kappa}$ is anti-symmetric then when we must have \[B(\tilde{u},\tilde{u})=(B\tilde{u},\tilde{u})=(\kappa\tilde{u},\tilde{u})+\big((B-\kappa)\tilde{u},\tilde{u})=(\kappa\tilde{u},\tilde{u})\in\mathbb{Z}\] for every $\tilde{u}\in\tilde{I}$, by anti-symmetry and the fact that $\tilde{\kappa}(\tilde{I}) \subseteq I$ and $(\tilde{I},I)\subseteq\mathbb{Z}$. On the other hand, the usual relation between quadratic and bilinear forms implies that $B(\tilde{u},\tilde{v})\in\frac{1}{2}\mathbb{Z}$ for every $\tilde{u}$ and $\tilde{v}$ in $\tilde{I}$, so that $B$ is a symmetric element of $\mathrm{Hom}\big(\tilde{I},\frac{1}{2}I\big)$. We can then take a basis for $\tilde{I}$ over $\mathbb{Z}$ and the dual basis for $I$, so that the matrix representing $B$ as an element of $\mathrm{Hom}\big(\tilde{I},\frac{1}{2}I\big)$ is the Gram matrix of $B$ as a bilinear form in this basis of $\tilde{I}$, which is therefore symmetric and has half-integral entries but integral diagonal entries. Since in this presentation elements of $\mathrm{Hom}(\tilde{I},I)$ are precisely those that are represented by integral matrices, and from matrix of the form thus described we can subtract an integral matrix and obtain an anti-symmetric result (which will represent an anti-symmetric homomorphism), this proves the other direction. Now, it is clear that altering $\tilde{\kappa}$ by an element of $\mathrm{Hom}^{as}(\tilde{I},I)$ does not affect the anti-symmetry of $B-\tilde{\kappa}$. On the other hand, if two such homomorphisms have the desired properties then their difference must both be in $\mathrm{Hom}(\tilde{I},I)$ and be anti-symmetric, and it is therefore in $\mathrm{Hom}^{as}(\tilde{I},I)$. This proves the lemma.
\end{proof}
It follows from Lemma \ref{symas} that when the condition on $B$ is satisfied, the set of the possible resulting anti-symmetric maps form a coset
\begin{equation}
c_{B}\!:=\!\big\{\eta\in\mathrm{Hom}^{as}\big(I^{*},\tfrac{1}{2}I\big)\big|\;B-\tilde{\eta}\in\mathrm{Hom}(\tilde{I},I)\big\}\!\in\!\mathrm{Hom}^{as}\big(I^{*},\tfrac{1}{2}I\big)\big/\mathrm{Hom}^{as}(I^{*},I), \label{coset}
\end{equation}
under the map $\eta\mapsto\tilde{\eta}$ from Equation \eqref{tildeiso}.

\smallskip

The characterization of $\Gamma_{L,I}$ in these coordinates can now be determined.
\begin{prop}
The element $A\in\mathcal{P}_{U}^{0}$ that is associated with the parameters $(M,\gamma,\psi,\eta)$ from above lies in $\Gamma_{L,I}$ if and only if the following conditions are satisfied:
\begin{enumerate}[$(i)$]
\item $\gamma\in\Gamma_{\Lambda}$.
\item $M$ is in $\mathrm{SL}(I_{L^{*}},I)$. In particular, $\mathrm{Id}_{U}-M:U/I_{L^{*}} \to U/I$ is well-defined.
\item $\psi(\Lambda) \subseteq I$, and the induced map from $\Delta_{\Lambda}$ to $U/I$ equals the composition of $\iota^{*}:\Delta_{\Lambda} \to U/I_{L^{*}}$ from Theorem \ref{latdecom} with $\mathrm{Id}_{U}-M:U/I_{L^{*}} \to U/I$ from part $(ii)$.
\item $\tilde{\eta}\in\mathrm{Hom}_{\mathbb{Q}}^{as}(\tilde{U},U)$ is $(\mathrm{Id}_{U}-M)\alpha\big(\mathrm{Id}_{\tilde{U}}-\widetilde{M^{*}}\big)+M\alpha-\alpha\widetilde{M^{*}}-\frac{\tilde{\psi}\widetilde{\psi^{*}}}{2}-\tilde{\kappa}$ for some $\kappa\in\mathrm{Hom}(I^{*},I)$.
\end{enumerate}
Moreover, $\gamma$ in condition $(i)$ can be arbitrary, for every $M\in\mathrm{SL}(I_{L^{*}},I)$ as in condition $(ii)$ there is some $\psi\in\mathrm{Hom}_{\mathbb{Q}}(W,U)$ satisfying condition $(iii)$, and for every such $M$ and $\psi$ there exists $\eta\in\mathrm{Hom}_{\mathbb{Q}}^{as}(U^{*},U)$ for which condition $(iv)$ is fulfilled (and then $\eta$ is unique up to $\mathrm{Hom}^{as}(I^{*},I)$). \label{paragrpZ}
\end{prop}

\begin{proof}
First, if $AL=L$ and $AU=U$ then $M=A\big|_{U}$ yields a $\mathbb{Z}$-automorphism of $I=U \cap L$, and if $\det M>0$ then $M\in\mathrm{SL}(I)$. Now, the condition $A\in\Gamma_{L}$ means that $A\lambda-\lambda \in L$ for every $\lambda \in L^{*}$, and we evaluate this difference for the two parts $\tilde{I}$ and $I^{\perp}_{L^{*}}$ of $L^{*}$ (see Theorem \ref{latdecom}). Consider first an element $\lambda \in I^{\perp}_{L^{*}}$, for which by Equations \eqref{decomL*} and \eqref{IperpL} we get $\lambda=u+w$ for \[w\in\tilde{\Lambda}^{*}\mathrm{\ and\ }u \in U\mathrm{\ with\ }u+I_{L^{*}}=-\iota^{*}\tilde{p}w,\mathrm{\ and\ then\ }A\lambda-\lambda \in I^{\perp}_{L}=I\oplus\tilde{\Lambda}.\] Since this holds for every such $u$ and $w$, it follows that
\begin{equation}
\tilde{\gamma}w-w\in\tilde{\Lambda}\ \mathrm{for\ any}\ w\in\tilde{\Lambda}^{*},\quad\mathrm{and}\quad Mu-u-\tilde{\psi}\tilde{\gamma}w \in I\ \mathrm{when}\ u+I_{L^{*}}=-\iota^{*}\tilde{p}w. \label{GammaLIperp}
\end{equation}
Condition $(i)$ thus immediately follows, and by taking $w=0$ (hence $u \in I_{L^{*}}$) in Equation \eqref{GammaLIperp} we deduce condition $(ii)$ as well. Now let $w$ be any element of $\tilde{\Lambda}$, which we write as $\tilde{\gamma}^{-1}v$ for $v\in\tilde{\Lambda}$, and since $\tilde{p}w=0$ we know that  $u \in I_{L^{*}}$ once again. Since $Mu-u \in I$ by condition $(ii)$, Equation \eqref{GammaLIperp} yields $\tilde{\psi}v \in I$ as well, proving the first part of condition $(iii)$. Considering $\psi\big|_{\Lambda^{*}}:\Lambda^{*} \to U$, the condition $\psi(\Lambda) \subseteq I$ produces a map $\psi^{\Delta}:\Delta_{\Lambda} \to U/I$, and then condition $(i)$ and Equation \eqref{GammaLIperp} imply that \[\psi^{\Delta}\tilde{p}w=\psi^{\Delta}(w+\tilde{\Lambda})=\psi^{\Delta}(\tilde{\gamma}w+\tilde{\Lambda})=Mu-u+I=(M-\mathrm{Id}_{U})u+I \in U/I.\] As condition $(ii)$ shows that the right hand side depends only on the image of $u$ in $U/I_{L^{*}}$, which was seen to be $-\iota^{*}\tilde{p}w$, we deduce that $\psi^{\Delta}=(\mathrm{Id}_{U}-M)\iota^{*}$, establishing condition $(iii)$.

Now, dualizing condition $(ii)$ and the isomorphisms $\tilde{I} \cong I^{*}$ and $\tilde{I}_{L}\cong(I_{L^{*}})^{*}$ imply that $\widetilde{M^{*}}$ and $\widetilde{M^{-*}}$ preserve both $\tilde{I}$ and $\tilde{I}_{L}$, and that we have \[\big(\widetilde{M^{*}}-\mathrm{Id}_{\tilde{U}}\big)(\tilde{I})\subseteq\tilde{I}_{L}\quad\mathrm{and}\quad\big(\widetilde{M^{-*}}-\mathrm{Id}_{\tilde{U}}\big)(\tilde{I})\subseteq\tilde{I}_{L}.\] Hence when we take $\lambda\in\tilde{I}$ and subtract it from the expression for $A\lambda$ in Equation \eqref{PUgenform}, the part $\tilde{u}=\widetilde{M^{-*}}\lambda-\lambda$ lies in $\tilde{I}_{L}$, as Equation \eqref{decomL} demands. Dualizing condition $(iii)$ via these isomorphisms implies that \[\widetilde{\psi^{*}}(\tilde{I})\subseteq\tilde{\Lambda}^{*},\quad\mathrm{and}\quad\tilde{p}\widetilde{\psi^{*}}:\tilde{I}\to\Delta_{\Lambda}\mathrm{\ equals\ }(\iota:\tilde{I}_{L}\to\Delta_{\Lambda})\circ(\mathrm{Id}_{\tilde{U}}-\widetilde{M^{*}}:\tilde{I}\to\tilde{I}_{L}).\] Since $\mu:=\widetilde{M^{-*}}\lambda\in\tilde{I}$, the second part $w=\widetilde{\psi^{*}}\mu=\widetilde{\psi^{*}}\widetilde{M^{-*}}\lambda$ from Equation \eqref{PUgenform} lies in $\tilde{\Lambda}^{*}$, and we also obtain the equality \[w+\tilde{\Lambda}=\tilde{p}w=\tilde{p}\widetilde{\psi^{*}}\mu=\tilde{p}\widetilde{\psi^{*}}\widetilde{M^{-*}}\lambda=\iota\big(\widetilde{M^{-*}}\lambda-\lambda\big)=\iota\tilde{u},\] as Equation \eqref{decomL} requires. Therefore, given conditions $(ii)$ and $(iii)$, the only additional requirement that remains in Equation \eqref{decomL} is the last one. Substituting the term for $u$ in Equation \eqref{PUgenform} in that condition reduces us to verify that the sum of \[M\alpha\lambda-\alpha\widetilde{M^{-*}}\lambda-\tfrac{\tilde{\psi}\widetilde{\psi^{*}}\widetilde{M^{-*}}\lambda}{2}-\tilde{\eta}\widetilde{M^{-*}}\lambda\quad\mathrm{and}\quad
2\alpha\tilde{u}=2\alpha\big(\widetilde{M^{-*}}\lambda-\lambda\big)\quad\mathrm{is\ in\ }I.\] Writing this in terms of $\mu=\widetilde{M^{-*}}\lambda$, and recalling the relation between $\kappa$ and $\eta$ in condition $(iv)$ (and with it the relation between the resulting maps $\tilde{\kappa}$ and $\tilde{\eta}$ as in Equation \eqref{tildeiso}) as well as that $\widetilde{M^{-*}}$ takes $\tilde{I}$ onto itself, this amounts to checking that
\begin{equation}
\tilde{\kappa}\mu=\Big[(M-\mathrm{Id}_{U})\alpha\big(\widetilde{M^{*}}-\mathrm{Id}_{\tilde{U}}\big)+M\alpha-\alpha\widetilde{M^{*}}-\tfrac{\tilde{\psi}\widetilde{\psi^{*}}}{2}-\tilde{\eta}\Big]\mu \in I\quad\mathrm{for\ every\ }\mu\in\tilde{I}. \label{condoneta}
\end{equation}
Hence $\eta$ must have the form required in condition $(iv)$, because $\tilde{\kappa}\in\mathrm{Hom}(\tilde{I},I)$ precisely when $\kappa\in\mathrm{Hom}(I^{*},I)$. Since all of these arguments are invertible, we have proved that $A\in\mathcal{P}_{U}^{0}$ lies in $\Gamma_{L}$, and therefore in $\Gamma_{L,I}$, if and only if our four conditions are satisfied.

Now, if $M\in\mathrm{SL}(I_{L^{*}},I)$ then the map $\iota\circ(\mathrm{Id}_{U}-M):\Delta_{\Lambda} \to U/I$ has finite image $J/I$ for some lattice $I \subseteq J \subseteq U$. Since $J$ is torsion-free, its composition with $p$ can be lifted to a map from $\Lambda^{*}$ to $J \subseteq U$ (as in the remark following Proposition \ref{choicecomp}), whose extension to an element $\psi\in\mathrm{Hom}_{\mathbb{Q}}(W,U)$ clearly satisfies condition $(iii)$.

Next, given both $M$ and $\psi$, showing the existence of $\eta$ satisfying condition $(iv)$ amounts to proving that there exists an element $\kappa\in\mathrm{Hom}(I^{*},I)$, or equivalently $\tilde{\kappa}\in\mathrm{Hom}(\tilde{I},I)$, such that the asserted formula for $\tilde{\eta}$ there is anti-symmetric. As duality and Lemma \ref{nonisocor} yield the equality \[\Big(\big(M\alpha-\alpha\widetilde{M^{*}}\big)\tilde{u},\tilde{v}\Big)=\big(\alpha\tilde{u},\widetilde{M^{*}}\tilde{v}\big)-\big(\alpha\widetilde{M^{*}}\tilde{u},\tilde{v}\big)
=\tfrac{1}{2}\big(\tilde{u},\widetilde{M^{*}}\tilde{v}\big)-\tfrac{1}{2}\big(\widetilde{M^{*}}\tilde{u},\tilde{v}\big)\] for $\tilde{u}$ and $\tilde{v}$ in $\tilde{U}$, the part $M\alpha-\alpha\widetilde{M^{*}}$ from Equation \eqref{condoneta} is always anti-symmetric. On the other hand, consider the map $(\mathrm{Id}_{U}-M)\alpha\big(\mathrm{Id}_{\tilde{U}}-\widetilde{M^{*}}\big)-\frac{\tilde{\psi}\widetilde{\psi^{*}}}{2}$ is a homomorphism from $\tilde{I}$ to $U$, as well as the resulting $\mathbb{Q}$-valued bilinear form on $\tilde{I}$, which we denote by $B$. First, duality and Lemma \ref{nonisocor} express $B(\tilde{u},\tilde{v})$ for $\tilde{u}$ and $\tilde{v}$ from $\tilde{U}$ as \[\Big(\!(\mathrm{Id}_{U}-M)\alpha\big(\mathrm{Id}_{\tilde{U}}-\widetilde{M^{*}}\big)\tilde{u}-\tfrac{\tilde{\psi}\widetilde{\psi^{*}}}{2}\tilde{u},\tilde{v}\!\Big)\!=\!\tfrac{1}{2}\!
\Big(\!\big(\mathrm{Id}_{\tilde{U}}-\widetilde{M^{*}}\big)\tilde{u},\!\big(\mathrm{Id}_{\tilde{U}}-\widetilde{M^{*}}\big)\tilde{v}\!\Big)-\frac{\big(\widetilde{\psi^{*}}\tilde{u},\widetilde{\psi^{*}}\tilde{v}\big)}{2}\!,\] which proves the symmetry. For applying Lemma \ref{symas} we must therefore verify that that $B(\tilde{u},\tilde{u})\in\mathbb{Z}$ for every $\tilde{u}\in\tilde{I}$, which by the last equality with $\tilde{v}=\tilde{u}$ amounts to to the equality of $\frac{[(\mathrm{Id}_{\tilde{U}}-\widetilde{M^{*}})\tilde{u}]^{2}}{2}$ and $\frac{[\widetilde{\psi^{*}}\tilde{u}]^{2}}{2}$ in $\mathbb{Q}/\mathbb{Z}$. But as $\widetilde{\psi^{*}}\tilde{u}\in\tilde{\Lambda}^{*}$ for $\tilde{u}\in\tilde{I}$, condition $(iii)$ on $\psi$ and the norm condition on $\iota$ in Theorem \ref{latdecom} imply that \[\tfrac{[\widetilde{\psi^{*}}\tilde{u}]^{2}}{2}+\mathbb{Z}=\tfrac{[\tilde{p}\widetilde{\psi^{*}}\tilde{u}]^{2}}{2}=\tfrac{[\iota(\mathrm{Id}_{\tilde{U}}-\widetilde{M^{*}})\tilde{u}]^{2}}{2}=
\tfrac{[(\mathrm{Id}_{\tilde{U}}-\widetilde{M^{*}})\tilde{u}]^{2}}{2}+\mathbb{Z}\] in $\mathbb{Q}/\mathbb{Z}$, as desired. Therefore Lemma \ref{symas} shows that we can indeed take $\kappa$ and $\tilde{\kappa}$ with the desired properties, and that the resulting map $\eta$ is unique up to $\mathrm{Hom}^{as}(I^{*},I)$ (note, however, that the resulting set of maps $\tilde{\eta}$ is \emph{not} the image of coset from Equation \eqref{coset} under the map from Equation \eqref{tildeiso}, but rather its translation by the anti-symmetric map $M\alpha-\alpha\widetilde{M^{*}}$). This completes the proof of the proposition.
\end{proof}

\smallskip

Considering now an element of $\Gamma_{L,I}\cap\mathcal{W}_{U}$, where conditions $(i)$ and $(ii)$ of Proposition \ref{paragrpZ} are immediate, condition $(iii)$ there, with $M=Id_{U}$, reduces to the inclusion $\psi(\Lambda^{*}) \subseteq I$. Since for such $\psi$ we get $\psi^{*}(I^{*})\subseteq\Lambda\subseteq\Lambda^{*}$ by dualizing, so that compositions like $\psi\varphi^{*}$ for two such maps $\varphi$ and $\psi$ lie in $\mathrm{Hom}(I^{*},I)$. We would therefore like to define a Heisenberg group over $\mathbb{Z}$, but we have to be careful, because the subset \[\mathrm{Hom}(\Lambda^{*},I)\times\mathrm{Hom}^{as}(I^{*},I)\subseteq\mathrm{Hom}_{\mathbb{Q}}(W,U)\times\mathrm{Hom}_{\mathbb{Q}}^{as}(U^{*},U)\] is not closed under the law of multiplication from Definition \ref{Heisdef}. Indeed, for such maps $\psi$ and $\varphi$, the combination $\frac{\psi\varphi^{*}-\varphi\psi^{*}}{2}$ from Proposition \ref{unipstruc} does not necessarily lie in $\mathrm{Hom}^{as}(I^{*},I)$. However, Proposition \ref{paragrpZ} gives us the desired form of the integral Heisenberg group, as it yields the following consequence.
\begin{cor}
For every $\psi\in\mathrm{Hom}(\Lambda^{*},I)$ denote by $c_{\psi}$ the coset from Equation \eqref{coset} that is associated with the bilinear form $-\frac{\tilde{\psi}\widetilde{\psi^{*}}}{2}$, namely $c_{\psi}$ consists of those $\eta\in\mathrm{Hom}^{as}\big(I^{*},\tfrac{1}{2}I\big)$ such that $\frac{\psi\psi^{*}}{2}+\eta\in\mathrm{Hom}(I^{*},I)$. Then \[\Gamma_{L}\cap\mathcal{W}_{U}=\big\{(\psi,\eta)\in\mathrm{Hom}(\Lambda^{*},I)\times\mathrm{Hom}^{as}\big(I^{*},\tfrac{1}{2}I\big)\big|\;\eta+\mathrm{Hom}^{as}(I^{*},I)=c_{\psi}\big\}.\] \label{ZHeis}
\end{cor}

\begin{proof}
Indeed, as mentioned above the parameter $\psi$ from Proposition \ref{paragrpZ} must lie in $\mathrm{Hom}(\Lambda^{*},I)$. In addition, since the difference $M\alpha-\alpha\widetilde{M^{*}}$ between the set of corresponding maps $\tilde{\eta}$ and the image of the coset from Equation \eqref{coset} under the map from Equation \eqref{tildeiso} vanishes when $M=Id_{U}$, and the bilinear form $B$ reduces to the asserted one in this case, we indeed obtain the corresponding subset. This proves the corollary.
\end{proof}
The fact that the subset from Corollary \ref{ZHeis} is a subgroup of $\mathcal{W}_{U}$, namely of the Heisenberg group $H\big(\mathrm{Hom}_{\mathbb{F}}(W,U),\mathrm{Hom}_{\mathbb{F}}^{as}(U^{*},U)\big)$, can also be seen directly. Indeed, for every $\psi\in\mathrm{Hom}(\Lambda^{*},I)$, the fact that $\psi^{*}(I^{*})$ is contained in the even lattice $\Lambda$ shows that for $u^{*} \in I^{*}$ the pairing $\big(\frac{\psi\psi^{*}u^{*}}{2},u^{*}\big)=\frac{(\psi^{*}u^{*},\psi^{*}u^{*})}{2}$ is in $\mathbb{Z}$, so that $c_{\psi}$ is indeed well-defined via Lemma \ref{symas}. In addition, if $\varphi$ is another element of $\mathrm{Hom}(\Lambda^{*},I)$ and we take $\eta \in c_{\psi}$ and $\rho \in c_{\varphi}$ then \[\tfrac{(\psi+\varphi)(\psi+\varphi)^{*}}{2}+\big(\eta+\rho+\tfrac{\psi\varphi^{*}-\varphi\psi^{*}}{2})=\big(\tfrac{\psi\psi^{*}}{2}+\eta\big)+\big(\tfrac{\varphi\varphi^{*}}{2}+\rho\big)+\psi\varphi^{*}
\in\mathrm{Hom}(I^{*},I),\] so that $\eta+\rho+\tfrac{\psi\varphi^{*}-\varphi\psi^{*}}{2} \in c_{\psi+\varphi}$, as desired. We therefore denote this subgroup by $H\big(\mathrm{Hom}(\Lambda^{*},I),\mathrm{Hom}^{as}(I^{*},I)\big)$ (compare with the definition of the group $\mathcal{P}(\Lambda)$ in Section 2.2 of \cite{[BL]}).

\smallskip

For presenting the main result, we recall the lattices $I=U \cap L$ and $\Lambda=I^{\perp}_{L}/I$ inside $W=U^{\perp}/U$, with dual $\Lambda^{*} \subseteq W$, as well as the group $\mathrm{SL}(I_{L^{*}},I)$ from Equation \eqref{SLJIdef}. We also write the (finite) image of $\iota^{*}$ in $U/I_{L^{*}}$ as $I_{\iota}/I_{L^{*}}$ for the appropriate lattice $I_{\iota} \subseteq U$, and note that \[\mathrm{the\ inclusion}\quad I_{L^{*}} \subseteq I_{\iota}\quad\mathrm{implies\ the\ inclusion}\quad\mathrm{SL}(I_{\iota},I)\subseteq\mathrm{SL}(I_{L^{*}},I).\] Using Proposition \ref{paragrpZ} we can now describe $\Gamma_{L,I}$ and the its image $\overline{\Gamma}_{L,I}$ in $\overline{\mathcal{P}}_{U}$, which is isomorphic to $\Gamma_{L,I}/\mathrm{Hom}^{as}(I^{*},I)$, as follows.
\begin{thm}
The map from Lemma \ref{unipker} restricts to a short exact sequence \[1\to\Gamma_{L,I}\cap\mathcal{W}_{U}=H\big(\mathrm{Hom}(\Lambda^{*},I),\mathrm{Hom}^{as}(I^{*},I)\big)\to\Gamma_{L,I}\to\mathrm{SL}(I_{L^{*}},I)\times\Gamma_{\Lambda}\to1,\] where the kernel is the group from Corollary \ref{ZHeis}. The group $\overline{\Gamma}_{L,I}$ sits in a similar short exact sequence \[1\to\mathrm{Hom}(\Lambda^{*},I)\to\overline{\Gamma}_{L,I}\to\mathrm{SL}(I_{L^{*}},I)\times\Gamma_{\Lambda}\to1,\] and after choosing a complement $\tilde{U}=\tilde{I}_{\mathbb{Q}}$ as in Lemma \ref{compoverZ}, with the map $\iota$ from Theorem \ref{latdecom}, the restriction of this short exact sequence to pre-images of $\mathrm{SL}(I_{\iota},I)\times\Gamma_{\Lambda}$ splits as a semi-direct product in the corresponding coordinates. \label{sdprodZ}
\end{thm}
In particular, the image of $\Gamma_{L,I}$ in the coarser quotient $\mathrm{GL}(U)$ from Corollary \ref{actquot} is $\mathrm{SL}(I_{L^{*}},I)$.

\begin{proof}
The form of the kernel $\Gamma_{L,I}\cap\mathcal{W}_{U}$ was proved in Corollary \ref{ZHeis}, and the short exact sequences are immediate consequences of Proposition \ref{paragrpZ}. For the splitting of the second one we simply observe that elements $\psi\in\mathrm{Hom}(\Lambda^{*},I)$ satisfy condition $(iii)$ of Proposition \ref{paragrpZ} if and only if $M$ is in the subgroup $\mathrm{SL}(I_{\iota},I)$ (by the definition of the latter group). This proves the theorem.
\end{proof}
Note that the pre-image of $\mathrm{SL}(I_{\iota},I)\times\Gamma_{\Lambda}$ in $\Gamma_{L,I}$ itself does not split the first short exact sequence from Theorem \ref{sdprodZ}, and also not in the special case considered in Corollary \ref{grpdim2} below. Indeed, while Corollary \ref{ZHeis} and the proof of Proposition \ref{paragrpZ} shows that when $M\in\mathrm{SL}(I_{\iota},I)$ and $\psi\in\mathrm{Hom}(\Lambda^{*},I)$ the change in the bilinear form $B$ may only modify the cosets $c_{\psi}$ in $\mathrm{Hom}^{as}\big(I^{*},\tfrac{1}{2}I\big)\big/\mathrm{Hom}^{as}(I^{*},I)$, in the formula for $\tilde{\eta}$ we also have the anti-symmetric part $M\alpha-\alpha\widetilde{M^{*}}$, which needs not be in $\mathrm{Hom}^{as}\big(I^{*},\frac{1}{2}I\big)$ at all. We remark that a direct verification shows that replacing $\tilde{U}$ by another complement $\hat{U}=\hat{I}_{\mathbb{Q}}$ as in Proposition \ref{choicecomp} preserves the description from Proposition \ref{paragrpZ} and Theorem \ref{sdprodZ}.

\smallskip

We would like to investigate the second short exact sequence from Theorem \ref{sdprodZ} a bit more. Any element $M\in\mathrm{SL}(I_{L^{*}},I)$ determines the class
\begin{equation}
b_{M}:=(\mathrm{Id}_{U}-M)\circ\iota^{*}\in\mathrm{Hom}(\Delta_{\Lambda},U/I)=\mathrm{Hom}(\Lambda,I)/\mathrm{Hom}(\Lambda^{*},I), \label{cocycle}
\end{equation}
and an element $\psi\in\mathrm{Hom}_{\mathbb{Q}}(W,U)$ satisfies condition $(iii)$ of Proposition \ref{paragrpZ} if and only if $\psi$ is (the extension of) an element of $\mathrm{Hom}(\Lambda,I)$ whose image modulo $\mathrm{Hom}(\Lambda^{*},I)$ is $b_{M}$. As this class is trivial if and only if $M\in\mathrm{SL}(I_{\iota},I)$, we find that elements of $\overline{\Gamma}_{L,I}$ with $\mathrm{GL}(U)$-images in $\mathrm{SL}(I_{L^{*}},I)\setminus\mathrm{SL}(I_{\iota},I)$ are paired with non-trivial cosets of $\mathrm{Hom}(\Lambda^{*},I)$ inside $\mathrm{Hom}(\Lambda,I)\subseteq\mathrm{Hom}_{\mathbb{Q}}(W,U)$, and the splitting of the short exact sequence does not extend further. The product rule from Corollary \ref{actquot} thus yields the following result.
\begin{prop}
The groups $\mathrm{SL}(I_{L^{*}},I)$ and $\Gamma_{\Lambda}$ operate naturally on the set $\mathrm{Hom}(\Delta_{\Lambda},U/I)$, with the action of the latter group being trivial. For $M$ and $N$ in the former group we have the cocycle condition $b_{MN}=b_{M}+M(b_{N})$. \label{actcoset}
\end{prop}

\begin{proof}
Assume that $A$ and $B$ are elements of $\overline{\Gamma}_{L,I}$, and Corollary \ref{actquot} shows that \[\mathrm{if\ }A\Longleftrightarrow(M,\gamma,\varphi)\mathrm{\ and\ }B\Longleftrightarrow(N,\delta,\varphi)\mathrm{\ then\ }AB\Longleftrightarrow(MN,\gamma\delta,\psi+M\varphi\gamma^{-1})\] in the coordinates $\mathrm{GL}(U)\times\mathrm{O}(W)\times\mathrm{Hom}_{\mathbb{Q}}(W,U)$. Therefore the action of $(M,\gamma)$ on $\mathrm{Hom}(\Delta_{\Lambda},U/I)$ is by composition with $M\in\mathrm{SL}(I_{L^{*}},I)\subseteq\mathrm{SL}(I)$ on the left (which is thus well-defined) and with $\gamma^{-1}$ on the right, the latter being trivial by the definition of $\Gamma_{\Lambda}$. Hence $b_{MN}$ is the coset containing the element $\psi+M\varphi\gamma^{-1}$, which is indeed the asserted one, in correspondence with the equality
\[b_{MN}=(\mathrm{Id}_{U}-MN)\circ\iota^{*}=(\mathrm{Id}_{U}-M)\circ\iota^{*}+M\circ(\mathrm{Id}_{U}-N)\circ\iota^{*}=b_{M}+M(b_{N}).\] This proves the proposition.
\end{proof}
Note that the map from Proposition \ref{actcoset} is \emph{not} a group homomorphism from $\mathrm{SL}(I_{L^{*}},I)$ to $\mathrm{Hom}(\Delta_{\Lambda},U/I)$ in general, and indeed, the set of $M\in\mathrm{SL}(I_{L^{*}},I)$ with $b_{M}=0$ was seen in Theorem \ref{sdprodZ} to be the subgroup $\mathrm{SL}(I_{\iota},I)$, and it is not necessarily normal in $\mathrm{SL}(I_{L^{*}},I)$.

\smallskip

Consider again the case in which $\iota=0$, where $L$ and $L^{*}$ are described in the simpler Equation \eqref{simpdecom}. Then $I_{\iota}=I_{L^{*}}$ (hence $\mathrm{SL}(I_{\iota},I)=\mathrm{SL}(I_{L^{*}},I)$), the cocycle from Equation \eqref{cocycle} is trivial, and the full group $\overline{\Gamma}_{L,I}$ splits as a semi-direct product in Theorem \ref{sdprodZ} (though $\Gamma_{L,I}$ does not in general). On the other hand, recall the simpler form of $\mathcal{P}_{U}$ appearing in Corollary \ref{stdim1} when $\dim U=1$. The structure of $\Gamma_{L,I}$ in this case is also much simpler, regardless of $\iota$.
\begin{cor}
Let $I$ be the subgroup of $L$ that is generated by the primitive isotropic vector $z \in L$, and set $U=I_{\mathbb{Q}}=\mathbb{Q}z \subseteq V$. Then $\Gamma_{L,I}$ is isomorphic to the semi-direct product in which $\Gamma_{\Lambda}$ operates on the group $(\Lambda,+)$. \label{dim1}
\end{cor}

\begin{proof}
We recall from Corollary \ref{stdim1} that $\mathcal{P}_{U}=\overline{\mathcal{P}}_{U}$ when $\dim U=1$, so that we can describe $\Gamma_{L,I}$ itself by the second exact sequence from Theorem \ref{sdprodZ}. Moreover, the group $\mathrm{SL}(I)$ is trivial when $I$ is of rank 1 (in correspondence with $\mathbb{Q}^{\times}_{+}$ having no non-trivial integral points), so that the sequence already splits. Moreover, when the generator $z$ for $U$ chosen for the isomorphism from Equation \eqref{isodim1} spans $I$ over $\mathbb{Z}$, this isomorphism identifies the kernel $\mathrm{Hom}(\Lambda^{*},I)$ from that sequence with $\Lambda$. This proves the corollary.
\end{proof}
The simple form of Corollary \ref{dim1} is one reason why we insisted on defining $\Gamma_{L,I}$ to be contained in the connected component $\mathcal{P}_{U_{\mathbb{R}}}^{0}$, as otherwise some elements mapping to $-\mathrm{Id}_{I}$ may lie in the discriminant kernel, and additional technical analysis will be required. Note that in the 1-dimensional case $I=\mathbb{Z}z$ the lattice $\Lambda$ is typically denoted by $K$, so that Corollary \ref{dim1} reproduces the semi-direct product of $\Gamma_{K}$ and $(K,+)$ appearing in \cite{[Bo]}, \cite{[Br]}, \cite{[Z1]}, and \cite{[Z2]}, among others.

As for the case with $\dim U=2$ which will be required for the toroidal boundary components below, we again take a basis for $U$ over $\mathbb{Q}$ that spans $I$ over $\mathbb{Z}$. This restricts the isomorphisms from Equation \eqref{isodim2} to isomorphisms
\begin{equation}
\mathrm{SL}(I)\cong\mathrm{SL}_{2}(\mathbb{Z}),\ \ \mathrm{Hom}^{as}(I^{*},I)\cong\mathbb{Z},\mathrm{\ and\ }\mathrm{Hom}(\Lambda^{*},I)\cong(\Lambda^{*})^{*}\times(\Lambda^{*})^{*}\cong\Lambda\times\Lambda, \label{SLISL2Z}
\end{equation}
where the cyclic group in the middle is contained as an index 2 subgroup of $\mathrm{Hom}^{as}\big(I^{*},\frac{1}{2}I\big)\cong\frac{1}{2}\mathbb{Z}$. The isomorphism in the middle of Equation \eqref{SLISL2Z} is the anti-symmetric part of the isomorphism \[\mathrm{Hom}(I^{*},I) \cong I \otimes I\cong\mathrm{M}_{2}(\mathbb{Z})\quad\mathrm{inside}\quad\mathrm{Hom}(I^{*},U) \cong I \otimes U\cong\mathrm{M}_{2}(\mathbb{Q}),\] and if $\psi\in\mathrm{Hom}(\Lambda^{*},I)$ is associated with the pair $(\lambda,\mu)\in\Lambda\times\Lambda$ then $\frac{\psi\psi^{*}}{2}$ is taken to the matrix $\frac{1}{2}\Big(\begin{smallmatrix} (\lambda,\lambda) & (\lambda,\mu) \\ (\mu,\lambda) & (\mu,\mu)\end{smallmatrix}\Big)\in\mathrm{M}_{2}(\mathbb{Q})$. It thus follows that the class $c_{\psi}$ from Corollary \ref{ZHeis} is the trivial class in $\mathrm{Hom}^{as}\big(I^{*},\frac{1}{2}I\big)\big/\mathrm{Hom}^{as}(I^{*},I)\cong\frac{1}{2}\mathbb{Z}\big/\mathbb{Z}$ when the pairing $(\lambda,\mu)$ is even, but the non-trivial one in case it is odd. Therefore the image of the group $H\big(\mathrm{Hom}(\Lambda^{*},I),\mathrm{Hom}^{as}(I^{*},I)\big)$ from Corollary \ref{ZHeis} inside the group $\tilde{H}(W,\mathbb{Q})$ from Corollary \ref{stdim2} takes the form
\begin{equation}
\tilde{H}(\Lambda,\mathbb{Z})=\big\{(\lambda,\mu,t)\in\Lambda\times\Lambda\times\tfrac{1}{2}\mathbb{Z}\big|\;t\in\tfrac{(\lambda,\mu)}{2}+\mathbb{Z}\big\}, \label{rk2Heis}
\end{equation}
which we shall naturally denote by $\tilde{H}(\Lambda,\mathbb{Z})$. Moreover, the groups $\mathrm{SL}(I_{L^{*}},I)$ and $\mathrm{SL}(I_{\iota},I)$, as defined in Equation \eqref{SLJIdef}, become, under the first isomorphism from Equation \eqref{SLISL2Z}, congruence subgroups of $\mathrm{SL}_{2}(\mathbb{Z})$, which we denote by $\Gamma_{L^{*}}$ and $\Gamma_{\iota}$ respectively. For example, if $\alpha$ and $\beta$ are generators for $I$ such that $\frac{1}{N}\alpha$ and $\frac{1}{D}\beta$ generate $I_{L^{*}}$, with $D$ dividing $N$, then the group $\Gamma_{L^{*}}$ is the classical congruence subgroup $\Gamma_{1}^{0}(N,D)=\Gamma_{1}(N)\cap\Gamma^{0}(D)$. In correspondence with Corollary \ref{stdim2}, Theorem \ref{sdprodZ} then takes the following form.
\begin{cor}
If $U=I_{\mathbb{Q}}$ is a 2-dimensional isotropic subspace of $V$ then choosing a basis for $I$ yields the two short exact sequences \[1\to\tilde{H}(\Lambda,\mathbb{Z})\to\Gamma_{L,I}\to\Gamma_{L^{*}}\times\Gamma_{\Lambda}\to1\mathrm{\ and\ }1\to\Lambda\times\Lambda\to\overline{\Gamma}_{L,I}\to\Gamma_{L^{*}}\times\Gamma_{\Lambda}\to1,\] where the kernel $\tilde{H}(\Lambda,\mathbb{Z})$ in the first sequence is defined in Equation \eqref{rk2Heis}. A choice of a (unimodular) complement $\tilde{U}=\tilde{I}_{\mathbb{Q}}$ yields a splitting of the second sequence over $\Gamma_{\iota}\times\Gamma_{\Lambda}$, where $\iota$ is the map from Theorem \ref{latdecom}. \label{grpdim2}
\end{cor}
We complement Corollary \ref{grpdim2} with the observation that using isomorphisms like the last one in Equation \eqref{SLISL2Z}, we can view the cocycle from Equation \eqref{cocycle} as taking values in the group $(\Lambda^{*}\times\Lambda^{*})/(\Lambda\times\Lambda)=\Delta_{\Lambda}\times\Delta_{\Lambda}$ (in fact, this group is isomorphic, via a $\mathbb{Z}$-basis for $I$, to $\Delta_{\Lambda}^{\dim U}$ for $U$ of any dimension), but its behavior does not simplify more than in the general case.

\section{Canonical Boundary Components of Toroidal Compactifications \label{TorComp}}

As an application of our analysis, which was also its original motivation, we determine the exact structure of the canonical boundary components of toroidal compactifications of orthogonal Shimura varieties. Let $L$ be an even lattice of signature $(n,2)$ in the quadratic space $V=L_{\mathbb{Q}}$. Then the symmetric space \[G(V_{\mathbb{R}})=\big\{v_{-} \subseteq V\big|\;v_{-}<<0,\dim v_{-}=2\}\] of $\mathrm{O}(V_{\mathbb{R}})$ carries a structure of an $n$-dimensional complex manifold, which is natural up to complex conjugation, as described in Section 13 of \cite{[Bo]}, Section 3.2 of \cite{[Br]}, Section 1.2 of \cite{[Z1]}, Section 3.1 of \cite{[F]} (all in the opposite signature), or Section 3 of \cite{[BZ]}. For the theory of toroidal compactifications we refer to \cite{[AMRT]} in general, to \cite{[Nam]} for the symplectic case, and to Sections 3 and 5 of \cite{[F]} or Sections 2 and 3 of \cite{[BZ]} for our orthogonal case. The fiber of the toroidal compactification over a 0-dimensional cusp, which is related to the group $\Gamma_{L,I}$ for $I \subseteq L$ of rank 1 (which has the simpler structure from Corollary \ref{dim1}), is not canonical in general, and depends on some choice of fan. On the other hand, over a 1-dimensional cusp lies a canonical toroidal boundary component, which is described grossly in Section 5 of \cite{[F]}, as well as precisely in Section 3 of \cite{[BZ]} under some simplifying assumptions, as an open variety of Kuga--Sato type. The goal of this section is to give the exact description of this boundary component without the simplifying assumption on $L$, $U$, and $I$ that appear in Hypothesis 3.14 of \cite{[BZ]} (note that this is not yet the full divisor on the toroidal compactification, and the form of the non-generic part of this divisor does depend on the choices of fans over 0-dimensional cusps).

For this let $I$ be a rank 2 isotropic lattice in $L$, set $U=I_{\mathbb{Q}}$, choose a basis $(z,w)$ for $I$ over $\mathbb{Z}$, and recall the coordinates from \cite{[K]} (or \cite{[F]}, or \cite{[BZ]}) that represent the complex manifold $G(V_{\mathbb{R}})$. Explicitly we have \[G(V_{\mathbb{R}})\!\cong\!\big\{Z_{V} \in V_{\mathbb{C}}\big|\;Z_{V}^{2}=0,\ (Z_{V},\overline{Z}_{V})<0,\ (Z_{V},z)=1,\ (\Re Z_{V},\!\Im Z_{V})\mathrm{\ oriented}\big\}\] for some fixed orientation on each $v_{-} \in G(V_{\mathbb{R}})$ that is determined by $(z,w)$ hence depends continuously on $v_{-}$. If $\tilde{I}$ is a complement for $I^{\perp}_{L^{*}}$ in $L^{*}$ that satisfies the condition of Lemma \ref{compoverZ} and $\tilde{U}=\tilde{I}_{\mathbb{Q}}$, then they are spanned over $\mathbb{Z}$ (resp. $\mathbb{Q}$) by the basis $(\zeta,\upsilon)$ that is dual to $(z,w)$, and by setting $\tilde{W}=(U\oplus\tilde{U}^{\perp})$ as in Equation \eqref{Wtilde} the set of $Z_{V} \in V_{\mathbb{C}}$ that represent $G(V_{\mathbb{R}})$ becomes
\begin{equation}
\Big\{Z_{V}=\zeta+\tau\upsilon+\tilde{Z}_{0}-\sigma w+\Big(\tau\sigma-\tfrac{\tilde{Z}_{0}^{2}+(\zeta+\tau\omega)^{2}}{2}\Big)z\Big|\;\Im\tau>0,\ \Im\sigma>\tfrac{(\Im\tau)^{2}\upsilon^{2}+(\Im\tilde{Z}_{0})^{2}}{2\Im\tau}\Big\}, \label{ZVcoor}
\end{equation}
where $\tau$ and $\sigma$ are complex numbers and $\tilde{Z}_{0}\in\tilde{W}_{\mathbb{C}}$. In particular, the coordinate $\tau$ from Equation \eqref{ZVcoor} lies in the upper half-plane $\mathcal{H}:=\{\tau\in\mathbb{C}|\;\Im\tau>0\}$, and $\sigma$ lies in a translated copy of $\mathcal{H}$ (note that this translation is bounded from below by $\Im\tau\frac{\upsilon^{2}}{2}$ over $\tau\in\mathcal{H}$, since $\tilde{W}_{\mathbb{R}} \cong W_{\mathbb{R}}$ is positive definite in our convention). The projection modulo $U_{\mathbb{C}}$ omits the coordinate $\sigma$, and yields images in
\begin{equation}
\widetilde{\mathcal{D}}(\Upsilon):=\bigcup_{\tau\in\mathcal{H}}W_{\mathbb{C}}^{1,\tau},\quad\mathrm{with}\quad W_{\mathbb{C}}^{1,\tau}:=\big\{\xi \in V_{\mathbb{C}}\big|\;(\xi,z)=1,\ (\xi,w)=\tau\big\}\big/U_{\mathbb{C}}, \label{tildeDF}
\end{equation}
independently of the choice of $\tilde{I}$ and $\tilde{U}$. Here and throughout $\Upsilon$ is the Baily--Borel cusp corresponding to $U$, which is isomorphic to $\mathcal{H}$, as well as the base space $\mathcal{D}(\Upsilon)$ from \cite{[AMRT]} and others. This produces the following description of $G(V_{\mathbb{R}})$ (as in, e.g., Proposition 3.12 of \cite{[BZ]}).
\begin{prop}
The symmetric space $G(V_{\mathbb{R}})$ is an affine $\mathcal{H}$-bundle over the space $\widetilde{\mathcal{D}}(\Upsilon)$ from Equation \eqref{tildeDF}, which itself carries a structure of a holomorphic affine vector bundle over $\mathcal{H}$. \label{GVRcoor}
\end{prop}
Indeed, the affine vector bundle structure in Proposition \ref{GVRcoor} is obtained via the natural projection sending $\xi$ in some $W_{\mathbb{C}}^{1,\tau}$ to the corresponding $\tau\in\mathcal{H}$.

Consider now the action of the group $\Gamma_{L,I}\cap\mathcal{W}_{U}=\tilde{H}(\Lambda,\mathbb{Z})$ from Equation \eqref{rk2Heis} and its subgroup $\mathrm{Hom}^{as}(I^{*},I)\cong\bigwedge^{2}I\cong\mathbb{Z}$, which are $\mathcal{W}_{\mathbb{Z}}(\Upsilon)$ and $\mathcal{U}_{\mathbb{Z}}(\Upsilon)$ respectively in the notation of \cite{[BZ]} (and others). Examining the action of the latter yields the following consequence of Proposition \ref{GVRcoor}, appearing in Corollary 3.17 of \cite{[BZ]}.
\begin{cor}
The quotient $\mathcal{U}_{\mathbb{Z}}(\Upsilon) \backslash G(V_{\mathbb{R}})$ is a punctured disc bundle over the space $\widetilde{\mathcal{D}}(\Upsilon)$. \label{puncdisc}
\end{cor}

\begin{proof}
One easily verifies that elements of $\mathcal{U}_{\mathbb{Z}}(\Upsilon)$ operate by addition on the coordinate $\sigma$, and dividing a translated upper half-plane by $\mathbb{Z}$ gives a punctured disc (under the exponential map). This proves the corollary.
\end{proof}

The coordinate on the fibers of the map from Corollary \ref{puncdisc} is the one denoted by $q_{2}$ in \cite{[K]}. The toroidal boundary component in which we are interested is obtained by filling in the zero section of this disc bundle, and it is therefore isomorphic to the image of $\widetilde{\mathcal{D}}(\Upsilon)$ under the action of $\overline{\Gamma}_{L,I}$ (see Proposition 3.19 of \cite{[BZ]}). The first step of determining this image is dividing by the kernel $\Lambda\times\Lambda$ of the associated short exact sequence from Corollary \ref{grpdim2} (this is the group $\mathcal{W}_{\mathbb{Z}}(\Upsilon)/\mathcal{U}_{\mathbb{Z}}(\Upsilon)$, denoted by $\mathcal{V}_{\mathbb{Z}}(\Upsilon)$ in \cite{[AMRT]} and others), and for describing the result, which is also contained in Corollary 3.17 of \cite{[BZ]}, we recall that \[\mathcal{E}\to\mathcal{H}\mathrm{\ is\ the\ universal\ elliptic\ curve,\ with\ fiber\ }E_{\tau}:=\mathbb{C}/(\mathbb{Z}\oplus\mathbb{Z}\tau)\mathrm{\ over\ }\tau\in\mathcal{H}.\]
\begin{prop}
The quotient $\mathcal{V}_{\mathbb{Z}}(\Upsilon)\backslash\widetilde{\mathcal{D}}(\Upsilon)$ is a principal homogenous space of the universal family $\mathcal{E}\otimes\Lambda$ over $\mathcal{H}$. It carries the punctured disc bundle $\mathcal{W}_{\mathbb{Z}}(\Upsilon) \backslash G(V_{\mathbb{R}})$. \label{ELambdaH}
\end{prop}

\begin{proof}
It is easy to check that the action of $\mathcal{V}_{\mathbb{Z}}(\Upsilon)\cong\Lambda\times\Lambda$ from Corollary \ref{grpdim2} becomes the additive action of the lattice $\Lambda\oplus\Lambda\tau$ on each fiber $W_{\mathbb{C}}^{1,\tau}$ of the second map from Proposition \ref{GVRcoor}. This proves the first assertion, since $W_{\mathbb{C}}^{1,\tau}$ is an affine model, or a principal homogenous space, of the complex vector space $W_{\mathbb{C}}=\Lambda_{\mathbb{C}}$. The second one follows from Corollary \ref{puncdisc}. This proves the proposition.
\end{proof}
As for the coordinate $q_{2}$ of the fibers punctured disc bundle, the action of an element of $\tilde{H}(\Lambda,\mathbb{Z})$, given by the parameters from Equation \eqref{rk2Heis}, takes the coordinate $\sigma$ from Equation \eqref{ZVcoor} to $\sigma+(\lambda,\tilde{Z}_{0})+\tau\frac{\lambda^{2}}{2}+\frac{(\lambda,\mu)}{2}+t$. This shows that the inequality from that equation is preserved (since $\tilde{Z}_{0}$ is sent to $\tilde{Z}_{0}+\tau\lambda+\mu$), but illustrates that the behavior of the coordinate $\sigma$ (and with it $q_{2}$) is more complicated. This is another indication of the fact that the short exact sequence involving $\Gamma_{L,I}$ in Proposition \eqref{sdprodZ} and Corollary \ref{grpdim2} does not split. Hence from now on we only consider the boundary component itself, which is defined by the equation $q_{2}=0$, an equation that is invariant under $\Gamma_{L,I}$.

\smallskip

The preliminary version of the form of the toroidal boundary component that lies over $\Upsilon$, which is the precise one if $\iota=0$ and $\Gamma_{\Lambda}$ is trivial (the latter always happens when $\Gamma_{L}$ is assumed to be neat), and gives a finite cover of the precise one in general, is as follows. We denote by $W_{L^{*}}^{\Lambda}$ the open Kuga--Sato variety that lies over $\Gamma_{L^{*}}\backslash\mathcal{H}$, in which the fiber over an element $\Gamma_{L^{*}}\tau$ is isomorphic to $E_{\tau}\otimes\Lambda$, and if $\tilde{I}$ and $\tilde{U}$ are chosen and $\iota$ is the map from Theorem \ref{latdecom}, then $W_{\iota}^{\Lambda}$ is the similarly defined open Kuga--Sato variety over $\Gamma_{\iota}\backslash\mathcal{H}$.
\begin{prop}
Consider the pre-image of $\Gamma_{\iota}\times\{\mathrm{Id}_{\Lambda}\}\subseteq\mathrm{SL}(I)\times\Gamma_{\Lambda}$ in $\overline{\Gamma}_{L,I}$. The quotient of $\widetilde{\mathcal{D}}(\Upsilon)$ by this group is a principal homogenous space over $W^{\Lambda}_{\iota}$. \label{fincover}
\end{prop}

\begin{proof}
The statement follows from the structure of this pre-image as a semi-direct product over $\mathbb{Z}$, proved in Theorem \ref{sdprodZ} and Corollary \ref{grpdim2}. This proves the proposition.
\end{proof}

The structure of a principal homogenous space means that when we choose the complements $\tilde{I}$ and $\tilde{U}$, each affine space $W_{\mathbb{C}}^{1,\tau}$ becomes naturally isomorphic to the vector space $\tilde{W}_{\mathbb{C}}$, by subtracting $\zeta+\tau\omega$ from Equation \eqref{ZVcoor}. Hence after dividing the fiber over $\tau$ by $\Lambda\oplus\Lambda\tau$ we obtain in Proposition \ref{fincover} a well-defined zero section for the projection from $W^{\Lambda}_{\iota}$ onto $\Gamma_{\iota}\backslash\mathcal{H}$. As the general statement in Theorem \ref{boundcomp} shows, such a zero section is not well-defined in general. Recall that each element $M\in\mathrm{SL}(I_{L^{*}},I)$ defines the class $b_{M}$ from Equation \eqref{cocycle}, and that when $\dim U=2$ (hence $M\in\Gamma_{L^{*}}$ as in Corollary \ref{grpdim2}), the class $b_{M}$ lies in $\Delta_{\Lambda}\times\Delta_{\Lambda}$. The cocycle condition from Proposition \ref{actcoset} allows us to define $W_{L^{*}}^{\Lambda,b}$ to be the quotient of the universal family $\mathcal{E}\otimes\Lambda$ from Proposition \ref{ELambdaH} under the action of $\Gamma_{L^{*}}$, in which an element $M$ in the latter group also acts on the fibers by translation by the image of $b_{M}$ in $\Delta_{\Lambda}\oplus\Delta_{\Lambda}\tau \subseteq E_{\tau}\otimes\Lambda$.

In addition, the group $\Gamma_{\Lambda}$ is finite (since $\Lambda$ is positive definite), and it acts on the space $W_{\mathbb{C}}$. The quotient $\Gamma_{\Lambda} \backslash W_{\mathbb{C}}$ is the affine algebraic variety $\mathrm{Spec}\big[(\mathrm{Sym}^{*}W_{\mathbb{C}})^{\Gamma_{\Lambda}}\big]$ (under the identification $W \cong W^{*}$ from Equation \eqref{isobil}), and as $\Gamma_{\Lambda}$ operates on the fibers of the map $\mathcal{E}\otimes\Lambda\to\mathcal{H}$, the structure sheaf of every fiber of the quotient is obtained by taking the $\Gamma_{\Lambda}$-invariant functions in the structure sheaf of $E_{\tau}\otimes\Lambda$. After dividing by $\Gamma_{\iota}$ we obtain a well-defined quotient $\Gamma_{\Lambda} \backslash W_{\iota}^{\Lambda}$ (with a similar structure sheaf), and since $\Gamma_{\Lambda}$ operates trivially on $\Delta_{\Lambda}$ as well as on the group $\mathrm{Hom}(\Delta_{\Lambda},U/I)$ from Equation \eqref{cocycle} and Proposition \ref{actcoset}, the quotient $\Gamma_{\Lambda} \backslash W_{L^{*}}^{\Lambda,b}$ is also well-defined. Considering the actions of all the groups involved therefore yields the following descripton of our toroidal boundary component.
\begin{thm}
The quotient $\overline{\Gamma}_{L,I}\backslash\widetilde{\mathcal{D}}(\Upsilon)$, and with it the generic subset of the toroidal boundary divisor associated with the cusp $\Upsilon$, is isomorphic to $\Gamma_{\Lambda} \backslash W_{L^{*}}^{\Lambda,b}$. \label{boundcomp}
\end{thm}
Note that while $W_{L^{*}}^{\Lambda,b}$ does have a zero section that is defined up to a subgroup of $\Delta_{\Lambda}\times\Delta_{\Lambda}$, the more canonical definition of $\overline{\Gamma}_{L,I}\backslash\widetilde{\mathcal{D}}(\Upsilon)$ in Theorem \ref{boundcomp} does not have a well-defined zero section at all, since it is constructed from the affine vector bundle $\widetilde{\mathcal{D}}(\Upsilon)$ from Equation \eqref{tildeDF}.

\noindent\textsc{Einstein Institute of Mathematics, the Hebrew University of Jerusalem, Edmund Safra Campus, Jerusalem 91904, Israel}

\noindent E-mail address: zemels@math.huji.ac.il

\end{document}